\documentclass[11pt]{amsart}
\usepackage{hyperref}
\usepackage{amsfonts,amssymb,amsmath,amsthm,amscd,euscript,array,mathrsfs}
\usepackage[all]{xy}
\usepackage{bbm}
\usepackage{booktabs}

\newcommand{\F}{\mathbb F}

\newcommand{\R}{\mathbb R}
\newcommand{\qH}{\mathbb H}
\newcommand{\C}{\mathbb C}

\newcommand{\Z}{\mathbb Z}
\newcommand{\bls}{\backslash}

\newcommand{\bfB}{\mathbf{B}}

\newcommand{\bfM}{\mathbf{M}}
\newcommand{\bfH}{\mathbf{H}}
\newcommand{\bfG}{\mathbf{G}}
\newcommand{\bfK}{\mathbf{K}}

\newcommand{\cP}{\mathcal P}

\newcommand{\cI}{\mathcal I}

\newcommand{\cS}{\mathcal S}

\newcommand{\g}{\mathfrak}

\newcommand{\lag}{\langle}
\newcommand{\rag}{\rangle}
\newcommand{\eps}{\varepsilon}

\newcommand{\oline}{\overline}
\newcommand{\sseq}{\subseteq}

\newtheorem{theorem}{\textbf{Theorem}}[section]

\newtheorem{thm}[theorem]{\textbf{Theorem}}
\newtheorem{definition}[theorem]{\textbf{Definition}}
\newtheorem{prp}[theorem]{\textbf{Proposition}}
\newtheorem{lem}[theorem]{\textbf{Lemma}}

\newtheorem{example}[theorem]{\textbf{Example}}

\newtheorem{remark}[theorem]{\textbf{Remark}}
\newenvironment{rmk}{\begin{remark}\rmfamily\upshape}{\end{remark}}

\newenvironment{dfn}{\begin{definition}\rmfamily\upshape}{\end{definition}}

\usepackage{enumitem}
\setitemize{itemsep=.5mm,topsep=1.5mm,parsep=0pt,
partopsep=0pt}

\usepackage{stmaryrd}

\newcommand{\into}{\hookrightarrow}

\newcommand{\gl}{\mathfrak{gl}}

\newcommand{\tr}{\mathrm{tr}}

\newcommand{\sPD}{\mathscr{PD}}
\newcommand{\sS}{\EuScript{S}}
\newcommand{\sD}{\mathscr{D}}
\newcommand{\sP}{\mathscr{P}}
\newcommand{\sR}{\mathscr{R}}

\newcommand{\Hom}{\mathrm{Hom}}

\newcommand{\sfe}{\mathsf{e}}

\newcommand{\POk}{ P^{\mathrm{ip}}}
\newcommand{\boldth}{\boldsymbol\theta}
\newcommand{\mapvWA}{\psi}
\newcommand{\sfE}{\mathsf{E}}

\newcommand{\Aa}{\underline{a}}
\newcommand{\Bb}{\underline{b}}
\newcommand{\Cc}{\underline{c}}
\newcommand{\Dd}{\underline{d}}
\newcommand{\Gact}{\circlearrowleft}
\newcommand{\lambul}{\lambda^\bullet}

\title[Quadratic Capelli operators and Okounkov polynomials]{Quadratic Capelli operators and Okounkov  polynomials}
\author{Siddhartha Sahi}

\address{Siddhartha Sahi, 
Department of Mathematics, Rutgers University, 110 Frelinghuysen Rd, Piscataway, NJ 08854--8019.
}

\email{sahi@math.rutgers.edu}

\author{Hadi Salmasian}

\address{Hadi Salmasian, Department of Mathematics and Statistics, University of Ottawa,
585 King Edward Ave, Ottawa, Ontario
Canada K1N 6N5.
}

\email{hsalmasi@uottawa.ca}
\begin{document}

\begin{abstract}
Let $Z$ be the symmetric cone of $r \times r$ positive definite Hermitian matrices over a real division algebra $\F$. Then $Z$ admits a natural family of invariant differential operators -- the 
\emph{Capelli operators} $C_\lambda$ -- indexed by partitions $\lambda$ of length at most $r$, whose eigenvalues are specializations of Knop--Sahi interpolation polynomials.

 In this paper we consider a double fibration  
 $Y \longleftarrow X \longrightarrow Z$ where $Y$ is the Grassmannian of $r$-dimensional subspaces of $\F^n $ with $n \geq 2r$. Using this  we construct a family of invariant differential operators $D_{\lambda,s}$ on $Y$ that we refer to as \emph{quadratic} Capelli operators. Our main result shows that the eigenvalues of the $D_{\lambda,s}$ are specializations of  Okounkov interpolation polynomials.\\[2mm]

\noindent\textsc{R\'{e}sum\'{e}}. 
Soit $Z$ le c\^{o}ne sym\'{e}trique de 
matrices de tailles $r\times r$ hermitiennes positives sur une alg\`{e}bre de division r\'{e}elle $ \F $. Alors $Z$ admet une famille naturelle d'op\'{e}rateurs diff\'{e}rentiels invariants - les
\emph{Op\'{e}rateurs de Capelli} 
$C_\lambda $ - index\'{e}s par des partitions 
$\lambda$ de longueur au plus $r$, dont les valeurs propres sont des sp\'{e}cialisations de polyn\^{o}mes d'interpolation Knop-Sahi.

  Dans cet article, nous considé\'{e}rons une double fibration
  $ Y \longleftarrow X \longrightarrow Z $ 
  o\`{u} $ Y $ est la vari\'{e}t\'{e} grassmannienne des sous-espaces de dimension $r$ de $\F^n$ avec 
  $n\geq 2r$. En utilisant cela, nous construisons une famille d'op\'{e}rateurs diff\'{e}rentiels invariants $ D_{\lambda, s}$ sur $ Y $ que nous appelons op\'{e}rateurs 
 Capelli \emph{quadratiques}. Notre r\'{e}sultat principal montre que les valeurs propres des 
 $ D_{\lambda, s} $ sont des sp\'{e}cialisations de polyn\^{o}mes d'interpolation Okounkov.\\[2mm]
\emph{Mot cl\'{e}:} 
Vari\'{e}t\'{e}s grassmanniennes, homomorphisme de Harish-Chandra,
polyn\^{o}mes Okounkov, 
op\'{e}rateurs Capelli quadratiques, c\^{o}nes sym\'{e}triques.

\end{abstract}

\maketitle

\section{Introduction}
\label{sec:introduction}
Let $\F=\R,\,\C,\,\qH$ be a real division algebra. Fix  integers $r$ and $n$ such that 
$1\leq r\leq \frac{n}{2}$. Let 
$Y$ be the Grassmannian of
$r$-dimensional subspaces of $\F^{n}$, and  let $Z$ be the symmetric
cone of $r\times r$ positive definite Hermitian $\F$-matrices. Then one has a
double fibration
\begin{equation*}\label{eqdblfib}
\xymatrix{
& X \ar[dl]_{\varphi}\ar[dr]^{\psi} & \\
Y & & Z
}
\end{equation*}
where $X$ is the space of $n\times r$ matrices of $\F$-rank $r$. For $x\in X$,
$\varphi\left(  x\right)  $ is the  column space (or range) of $x$, while $\psi\left(
x\right)  :=x^{\dagger}x$, where $x^{\dagger}$ denotes the $\F$-Hermitian adjoint
of $x$. 

One can give another description of the  above structure in terms of the groups 
\begin{equation*}
\label{ObVEmBB} 
G_{m}:=\mathrm{GL}_m(\F),\quad K_{m}:=\mathrm{U}_m(\F):=\left\{g\in G_m\ :\ g^\dagger g=I_{m\times m}\right\}.
\end{equation*}
The groups
$K_n$ and $G_r$ act on $X$ by matrix multiplication on the left and right respectively, and the maps  
$X\overset{\psi}{\longrightarrow}Z$
and
$X\overset{\varphi}{\longrightarrow}Y$ 
are simply the corresponding quotient maps.
Moreover,
$X\overset{\varphi}{\longrightarrow}Y$ is a principal $G_r$-bundle, while
 $X\overset{\psi}{\longrightarrow}Z$
is a fibration whose fibers are isomorphic to the Stiefel manifold $K_n/K_{n-r}$. 
Also, since the actions of $K_n$ and $G_r$ on $X$ commute, 
it follows that $G_r$  acts
 on $Z$, and $K_n$ acts on
 $Y$.
In fact, $Y$ and $Z$ are symmetric spaces
for the latter actions.
More precisely, 
we have 
\[
Y\simeq K_n/(K_r\times K_{n-r})
\,\text{ and }\,Z\simeq G_r/K_r.
\]

The cone $Z$ is a symmetric space of type $A$, and admits an important basis of $G_{r}$-invariant
differential operators $C_\lambda$, 
indexed by partitions $\lambda\in \cP_r$, where 
\[\cP_r:=
\{(\lambda_1,\ldots,\lambda_r)\in\Z^r\,:\,\lambda_1\geq\cdots\geq\lambda_r\geq 0\}.
\] The operators $C_\lambda$ were first studied by the first author in \cite{Sahi}, and were referred to as  
\emph{Capelli operators}.
It is known that the spectrum of $C_\lambda$ is given by
specialization of  Knop--Sahi type $A$ interpolation polynomials 
\cite{Sahi}, \cite{KnopSahi}, \cite{wallach}. 

On the other hand, $Y$ is a compact symmetric space of type $BC$. In this paper, we
use the above  double fibration  to construct a 
family of $K_{n}$-invariant differential operators 
$D_{\lambda,s}$ 
 on $Y$  that 
correspond
to the  Capelli operators 
$C_\lambda$.
We call the operators $D_{\lambda,s}$ the \emph{quadratic Capelli operators} because they are obtained from $C_\lambda$ by  pullback of the quadratic map $\psi$.
Our main result proves that the spectrum of  $D_{\lambda,s}$ is
given by specialization of the \emph{Okounkov} type $BC$ interpolation  
polynomials $P_\lambda(x;\tau,\alpha)$ 
(see  \cite[Sec. 5.3]{Koornwinder} and \cite{Okounkov}).

To describe our main result precisely, we begin by introducing some notation. Set $K:=K_n$ and 
$M:=K_{r}\times K_{n-r}\subset K$, so that $Y\simeq K/M$.
The group $K$ acts by left translation on 
$C^\infty(Y)$,
the space of complex-valued smooth functions on $Y$. 
The operators $D_{\lambda,s}$
leave the subspace 
$C^\infty(Y)_{K\text{-finite}}^{}$
of  
$K$-finite vectors 
invariant. 
By standard results from the theory of compact symmetric spaces (for example, see \cite[Chap. V]{Helg}), 
$C^\infty(Y)_{K\text{-finite}}^{}$ decomposes as  a 
multiplicity-free direct sum of irreducible 
$M$-spherical $K$-modules, which are naturally parametrized by partitions $\mu\in\cP_r$.
Our next goal is to describe this parametrization. 
Let $\g k$ and $\g m$ denote the Lie algebras of $K$ and $M$. Fix a Cartan decomposition 
$\g k=\g m\oplus\g p$. Let
 $\g a\sseq\g p$ be a 
Cartan subspace,  
  and let
 $\g h$ be a Cartan subalgebra of $\g k$ such that $\g a \sseq \g h$.
Then 
$\g h=\g t\oplus\g a$, where $\g t:=\g h\cap\g m$. 
Set  
 $\g k_\C:=\g k\otimes_\R\C$,
$\g h_\C:=\g h\otimes_\R\C$,  
$\g a_\C:=\g a\otimes_\R\C$, and 
$\g t_\C:=\g t\otimes_\R\C$.
 The restricted root system $\Sigma:=\Sigma(\g k_\C,\g a_\C)$  
is of type $BC_r$. We choose a positive system $\Sigma^+\subset \Sigma$ and a
 basis $\sfe_1,\ldots,\sfe_r$ of $\g a^*_\C$ such that  the multiplicity $m_\alpha$ of every $\alpha\in\Sigma^+$ is 
given in terms of $n$, $r$, and  $d:=\dim\F$ in Table 1 below.
 \vspace{2mm}

\begin{center}
\begin{tabular}{cccc}
\bottomrule
$\alpha$\  & \ $\sfe_i$, $1\leq i\leq r$ \ & \ $\sfe_i\pm \sfe_j$, $1\leq i< j\leq r$\  & \ $2\sfe_i$, $1\leq i\leq r$\ \\
\toprule
$m_\alpha$ \ & $d(n-2r)$ &  $d$ & $d-1$. \\
\bottomrule \end{tabular}

\vspace{3mm}
Table 1.
\end{center}
We also choose a  positive system for the root system $\Delta:=\Delta(\g k_\C,\g h_\C)$ which is compatible with $\Sigma^+$.  
Let $\widetilde \mu\in\g h_\C^*$.
By the Cartan--Helgason Theorem, $\widetilde{\mu}$  is the highest weight of an irreducible $M$-spherical $K$-module if and only if 
\begin{equation}
\label{eqwtiMUU}
\widetilde\mu\big|_{\g t_\C}=0
\text{ and }\widetilde\mu\big|_{\g a_\C^{}}=\sum_{i=1}^r2\mu_i\sfe_i,
\text{ where }
\mu:=(\mu_1,\ldots,\mu_r)\in\cP_r.
\end{equation} 
\begin{rmk}
Assume that $\F=\R$. Then $K$ is disconnected,
 and if $K^\circ$ denotes the connected component of identity of $K$, then  
$M\cap K^\circ$ is also disconnected. Therefore the Cartan--Helgason Theorem as stated for instance  in
\cite[Cor. V.4.2]{Helg} does not apply immediately to the case $\F=\R$.
However, one can use the refinement of the 
Cartan--Helgason Theorem for the pair $(K^\circ,M\cap K^\circ)$, given in   \cite[Sec. 12.3.2]{GW09}, as well as  the description of irreducible representations of $K$ in terms of irreducible representations of $K^\circ$, given in \cite[Thm 5.5.23]{GW09}, to obtain the condition \eqref{eqwtiMUU}.
\end{rmk}
From now on, we denote the $M$-spherical $K$-module with highest weight $\widetilde \mu$ satisfying
\eqref{eqwtiMUU} by $V_\mu$. Therefore as $K$-modules,
\[
C^\infty(Y)_{K\text{-finite}}^{}\simeq\bigoplus_{\mu\in\cP_r}
V_\mu.
\]
The operator $D_{\lambda,s}$ acts on $V_\mu$ by the scalar \[
c_{\lambda,s}(\mu):=\mathrm{HC}(D_{\lambda,s})\left(\widetilde\mu|_{\g a_\C}^{}+\rho\right),
\] where  
$\rho:=\frac{1}{2}\sum_{\alpha\in\Sigma^+}\alpha$, 
$\widetilde \mu$ is the highest weight of $V_\mu$, and 
$
\mathrm{HC}:\mathbf{D}_{K}(Y)\to \sP(\g a^*_\C)^{\mathsf W}
$ is the 
Harish-Chandra homomorphism from the algebra 
$\mathbf{D}_{K}(Y)$ of $K$-invariant differential operators on $Y$ onto the algebra of  polynomials on $\g a^*_\C$ that are invariant under the action of the 
restricted Weyl group $\mathsf W$. 

We now recall the definition of the Okounkov polynomials $P_\lambda(x;\tau,\alpha)$. 
Let $\Bbbk:=\C(\tau,\alpha)$ denote the field of rational functions in $\tau$ and $\alpha$.
Let $\delta,\mathbbm 1\in\cP_r$ and $\varrho_{\tau,\alpha}\in\Bbbk^r$ be defined by 
\begin{equation}
\label{rhotaualph}
\delta:=(r-1,\ldots,0),\quad
\mathbbm 1:=(1,\ldots,1),\quad 
\varrho_{\tau,\alpha}:=\tau\delta+\alpha\mathbbm 1.
\end{equation}
For $\lambda\in\cP_r$, we define $|\lambda|:=\sum_{i=1}^r\lambda_i$.
 Up to a scalar,  
$P_\lambda(x;\tau,\alpha)\in\Bbbk[x_1,\ldots,x_r]$ is the unique polynomial of degree $2|\lambda|$
  which is invariant under  permutations and sign changes  of $x_1,\ldots,x_r$, and satisfies
\[
P_\lambda(\mu+\varrho_{\tau,\alpha};\tau,\alpha)=0
\]
 for every $\mu\in\cP_r$ such that $|\mu|\leq |\lambda|\text{ and }\mu\neq \lambda$
(for more details, see Section \ref{sec:van-uniqn}).

Recall that $d:=\dim\F$. Let
$\mathbbm i:\C^r\to\g a_\C^*$ 
be the linear  map defined by $\mathbbm i(\sfe^i):= 2\sfe_i$ for $1\leq i\leq r$, where
$\sfe^1,\ldots,\sfe^r$ are the standard basis vectors of $\C^r$
 (therefore $\mathbbm i(\cP_r)$ is  the set of restrictions to $\g a_\C$ of highest weights of $M$-spherical $K$-modules). Set \[
\varrho:=\mathbbm i^{-1}(\rho).
\] 
A simple calculation yields
\begin{equation}
\label{formulRho}
\varrho=(\varrho_1,\ldots,\varrho_r)\,\text{ where }\,
\varrho_i:=\frac{dn}{4}-\frac{1}{2}-\frac{d(i-1)}{2}\,\text{ for every }1\leq i\leq r.
\end{equation}

We are now ready to state our main theorem. 
\begin{thm}
\label{maintheorem}
For every  $\lambda,\mu\in\cP_r$ and every $s\in\C$, the operator $D_{\lambda,s}$ acts on $V_\mu$ by the scalar 
\begin{equation}
\label{scallll}
c_{\lambda,s}(\mu):=\gamma_\lambda
P_\lambda\left(\mu+\varrho;\frac{d}{2},s-\varrho_1\right),
\end{equation}
where $\gamma_\lambda$ is a certain explicit constant, defined in \eqref{defConstgamla} below.
\end{thm}

We now briefly sketch the  strategy behind the proof of  Theorem \ref{maintheorem}. We first show that the equality \eqref{scallll} holds up to a scalar multiple. In view of the characterization of the Okounkov polynomials $P_\lambda$, it suffices to prove that $c_{\lambda,s}(x-\varrho)$ is a polynomial in $x$ and $s$ which has the same degree and vanishing property as
 $P_\lambda(x;\tau,\alpha)$ for $\tau:=\frac{d}{2}$ and $\alpha:=s-\varrho_1$ (see 
Proposition \ref{prpuniqPip} and Proposition \ref{PrpVanCondd}).
It will be seen that verifying the pertinent vanishing property  can be reduced to a slightly weaker one, that is, to show that  
 the operators $D_\lambda:=D_{\lambda,0}$ vanish on certain  $V_\mu\subset C^\infty(Y)_{K\text{-finite}}^{}$. 
Up to this point, the strategy is the same as the one in the case of Knop--Sahi type $A$ polynomials. However, the proof of the vanishing property of  
the operators $D_\lambda$ is more subtle than the type $A$ case, in that it does \emph{not} follow from a direct reasoning that is based on orders of differential operators. In addition, the \emph{$\rho$-shift} of the symmetric space $Y$ and  $\varrho_{\tau,\alpha}$ are not identical. Rather, they are related to each other as in Remark \ref{RHORHO}.

To overcome these difficulties, we need to use the fact that $c_{\lambda,s}(x-\varrho)$ is symmetric under permutations and sign changes of its variables, and therefore $c_{\lambda,s}(x-\varrho)=0$ for $x:=(x_1,\ldots,x_r)$ if and only if 
$c_{\lambda,s}(\tilde x-\varrho)=0$ 
for $\tilde x:=(-x_r,\ldots,-x_1)$.
The latter observation results in an equivalent form of the vanishing condition for $c_{\lambda,s}$, which is verified in the proof of  
Proposition \ref{PrpVanCondd} using differential operator techniques and branching rules. 
The branching rule from $G_n$ to $K_n$ that we need in the proof of the vanishing property of $D_\lambda$ is the Littlewood--Richardson Rule for $\F=\C$ and
the Littlewood Restriction Theorem for $\F=\R,\qH$. As a result, some of our proofs are divided into two cases, but they lead to uniform statements. 

The last step in the proof of our main result is to determine the scalar $\gamma_\lambda$ that relates the two sides of 
\eqref{scallll}. To this end, we use the fact that the top degree  homogeneous term of $P_\lambda(x;\tau,\alpha)$ is a Jack polynomial. Using a trick which relies on an identity for Jack polynomials (see \cite[Prop. 2.3]{StanleyP}), calculation of $\gamma_\lambda$ for general $\lambda$ is reduced to when $\lambda$ corresponds to the $1\times 1$ Young diagram. Even this special case requires a rather intricate computation to express the corresponding quadratic Capelli operator  in terms of the Casimir operators of $G_n$ and $K_n$. This is  carried out in the Appendix.

We  now describe the relation between
Theorem \ref{maintheorem} and earlier results on invariant differential operators on symmetric spaces.
The Capelli operators $C_\lambda$
were originally studied in connection to the 
  famous Capelli identity, which has also been  considerably generalized by  Howe and Umeda 
\cite{HoweUmeda} 
 from the viewpoint of 
 multiplicity-free actions, and by Kostant and the first author 
\cite{KostantSahi1}, \cite{KostantSahi2} 
 from the viewpoint of Jordan algebras.
For $\F=\R$ and $\lambda=\mathbbm 1$, the operator $D_{\mathbbm 1,-1}$ was first considered
 by Howe and Lee in \cite{HoweLee}, who computed its spectrum for $r=2$ and asked for the determination of the spectrum for general  $r$.  This was solved  by the first author more generally for $D_{m\mathbbm 1,s}$ and $\F=\R$,
\cite{SahiRep}, and subsequently 
by  Zhang and the first author  \cite{SahiZhang} 
 for arbitrary $\F$, where a connection with the Radon transform was also established.
The explicit form of the answer in \cite{SahiZhang} gave us the first hint that the general situation might have something to do with the Okounkov polynomials.
The result of \cite{SahiZhang} on the spectrum of $D_{m\mathbbm 1,s}$  is indeed a special case of our  Theorem \ref{maintheorem}.

Finally, we say a few words about prospects for future research that emerge from this work. 
Quite recently, Zhang and the first author  established another link between Okounkov interpolation polynomials and  the spectrum of Shimura operators on Hermitian symmetric spaces \cite{SahiZhang2}. 
It would be   interesting  to understand the connection between our main result and the results of \cite{SahiZhang2}.
Also, in view of our recent work \cite{SahiSal} on the Capelli eigenvalue problem in the case of the  supersymmetric pairs
$
(\g{gl}(m|m')\times \g{gl}(m|m'),\g{gl}(m|m'))$
 and $(\g{gl}(m|2m'),\g{osp}(m|2m'))$, 
it is natural to ask whether Theorem \ref{maintheorem} can also be extended to the setting of Lie superalgebras. This is likely to involve the deformed $BC$
interpolation polynomials of \cite{SerVes}. We remark that in forthcoming papers  \cite{AlldridgeSahiSal} and \cite{SahiSalSer}, we extend the results of 
\cite{SahiSal} to the setting of multiplicity-free actions obtained from Jordan superalgebras.
 Another interesting problem is to extend the  Littlewood Restriction Theorem (see Proposition \ref{thmLIL}) to the super setting, namely to
$(\g{gl}(m|2m'),\g{osp}(m|2m'))$. We are planning to study these problems in the near future. \\[2mm]
\textbf{Acknowledgement.} The authors thank Kyo Nishiyama and Nolan Wallach for helpful e-mail correspondences.
The research of Siddhartha Sahi was partially supported by a Simons Foundation grant  (509766) and of Hadi Salmasian by an NSERC Discovery Grant (RGPIN-2013-355464). 
Part of this work was carried out during the Workshop on Hecke Algebras and Lie Theory  held at the University of Ottawa during May 12--15, 2016. The authors thank the National Science Foundation (DMS-162350), the Fields Institute, and the University of Ottawa for funding this workshop.

\section{Parametrization of representations by partitions}
\label{Sec-PaamPart}
In this article, we will need various parametrizations of finite dimensional 
representations of 
$G_n$, $K_n$, and $G_r$ by partitions. Instead of working with representations of these real Lie groups, it will be more convenient to work with representations of  their complexifications.

Let $W:=\mathrm{Mat}_{n\times r}(\F)$ denote the space of $n\times r$ matrices with entries in $\F$. Furthermore, set
\[
A:=\{x\in \mathrm{Mat}_{r\times r}(\F)\, :\, x^\dagger=x\}.
\]
Then $X\subset W$ and $Z:=\{w^\dagger w\,:\, w\in X\}\subset A$ are open.
The $G_n\times G_r$-action on  $X$ is the restriction of the $G_n\times G_r$-action on $W$ given by
\[
(g_1,g_2)\cdot w:=g_1^{}wg_2^{-1}\text{ for }
(g_1,g_2)\in G_n\times G_r,\ w\in W.
\]
The $G_r$-action on $Z$ is the restriction of the $G_r$-action on $A$ given by 
\[
g\cdot x:=(g^\dagger)^{-1} xg^{-1}\ \text{ for }g\in G_r,\
x\in A.
\]
Let $\bfG_n$ and $\bfG_r$   
denote the complexifications of the real Lie groups $G_n$ and $G_r$. Similarly, let $\bfK_n\subset\bfG_n$ denote the complexification of $K_n$.
\begin{rmk}
\label{rmk-qHin2C}
From now on, we need to fix an embedding of quaternionic matrices into complex matrices of twice the size. For integers $k,m\geq 1$, let $x:=a+b\mathbf j\in\mathrm{Mat}_{k\times m}(\qH)$, where $a,b\in\mathrm{Mat}_{k\times m}(\C)$. We set
\[
\breve x:=\begin{bmatrix}
a& -\oline b\\
b & \oline a
\end{bmatrix}\in\mathrm{Mat}_{2k\times 2m}(\C).
\]
\end{rmk}
\begin{rmk}
The matrix realizations of the embeddings 
$\boldsymbol i_n:G_n\into \bfG_n$ and $\boldsymbol i_r:G_r\into \bfG_r$ are as follows. If 
$\F=\R$, then $\bfG_n\simeq \mathrm{GL}_n(\C)$ and
$\boldsymbol i_n:\mathrm{GL}_n(\R)\to \mathrm{GL}_n(\C)$ is the obvious map. If $\F=\C$, then $\bfG_n\simeq \mathrm{GL}_n(\C)\times \mathrm{GL}_n(\C)$ and $\boldsymbol i_n:\mathrm{GL}_n(\C)\to \mathrm{GL}_n(\C)\times \mathrm{GL}_n(\C)$ is the map 
$g\mapsto ((g^*)^{-1},g)$, where $g^*:=\oline g^T$ is the adjoint of $g$.
If $\F=\qH$, then 
$\bfG_n\simeq\mathrm{GL}_{2n}(\C)$ and 
$\boldsymbol i_n:\mathrm{GL}_n(\qH)\to \mathrm{GL}_{2n}(\C)$ is the map $g\mapsto \breve g$. The matrix realization of $\bfK_n$ as a subgroup of $\bfG_n$ is as follows. If $\F=\R$, then 
$\bfK_n=\{g\in\mathrm{GL}_n(\C)\, :\, g^Tg=I_{n\times n}\}$. If $\F=\C$, then $\bfK_n=\{(g,g)\, :\, g\in\mathrm{GL}_n(\C)\}$. Finally, if $\F=\qH$, then 
$\bfK_n=\{g\in\mathrm{GL}_{2n}(\C)\, :\, g^TJ_ng=J_n\}$, where
\begin{equation}
\label{eqMtrxJ}
J_n:=\begin{bmatrix}
0 & I_{n\times n}\\
-I_{n\times n} & 0
\end{bmatrix}.
\end{equation}
The definition of the embedding $\boldsymbol i_r:G_r\to \bfG_r$ is similar to that of $\boldsymbol i_n$. 
To help the reader, we summarize the information about $\bfG_n$ and $\bfK_n$  in Table 2 below.
\vspace{1mm}

\begin{center}
\begin{tabular}{cccc}
\bottomrule
$\F$ & $\bfG_n$ & $\bfK_n$ & Realization of $\bfK_n$ in $\bfG_n$\\
\toprule

$\R$ & $\mathrm{GL}_n(\C)$ & $\mathrm{O}_n(\C)$ & 
$\{g\in \mathrm{GL}_n(\C)\,:\,g^Tg=I_{n\times n}\}$\\

$\C$ &\ \  $\mathrm{GL}_n(\C)\times \mathrm{GL}_n(\C)$ \ \ & 
$\mathrm{GL}_n(\C)$ & $\{(g,g)\,:\,g\in\mathrm{GL}_n(\C)\}$\\

$\qH$ & $\mathrm{GL}_{2n}(\C)$ & $\mathrm{Sp}_{2n}(\C)$ &\ \  $\{g\in\mathrm{GL}_{2n}(\C)\,:\,g^TJ_ng=J_n\}$\ \ 
\\
\bottomrule
\end{tabular}

\vspace{2mm}
Table 2.\\[2mm]
\end{center}

\end{rmk}

Set $W_\C:=W\otimes_\R\C$ and $A_\C:=A\otimes_\R \C$.   
The map $X\overset{\mapvWA}{\longrightarrow}Z$
 has a unique holomorphic extension
$W_\C\overset{\mapvWA}{\longrightarrow}A_\C$.
 For an explicit description of  $A_\C$, $W_\C$, and the  map $\mapvWA:W_\C\to A_\C$, see Table 3 below.
\vspace{1mm}

  \begin{center}
\begin{tabular}{ccccc}
\bottomrule
$\F$  & $A_\C$ & $W_\C$ & $\mapvWA$ \\
\toprule
$\R$  & $\mathrm{Sym}_{r\times r}(\C)$ & $\mathrm{Mat}_{n\times r}(\C)$ & $x\mapsto x^Tx$\\
$\C$ &  
$\mathrm{Mat}_{r\times r}(\C)$ & $\mathrm{Mat}_{n\times r}(\C)\oplus\mathrm{Mat}_{n\times r}(\C) $ & $(x_1^{},x_2^{})\mapsto x_1^Tx_2^{}$\\
$\qH$&   $\mathrm{Skew}_{2r\times 2r}(\C)$ & $\mathrm{Mat}_{2n\times 2r}(\C)$ & $x\mapsto -x^TJ_nx$\\
\bottomrule
\end{tabular}

\vspace{2mm}
Table 3.\\[2mm]
\end{center}
In Table 3,  $\mathrm{Sym}_{r\times r}(\C)$ denotes the space of complex symmetric $r\times r$ matrices and  $\mathrm{Skew}_{2r\times 2r}(\C)$ denotes the space of complex skew symmetric $2r\times 2r$ matrices.
\begin{rmk}
The matrix realizations of the maps $A\to A_\C$ and 
$W\to W_\C$ are as follows. For $\F=\R$ and $\F=\C$, the map $A\into A_\C$ is the obvious embedding, and for $\F=\qH$, it is the map $a\mapsto -J_ra$, where 
$J_r$ is defined similar to \eqref{eqMtrxJ}.  
In fact the map $A\into A_\C$ is related to realization of Euclidean Jordan algebras (see
\cite[Sec. VIII.5]{FaKo}).
As for 
 $W\into W_\C$, it is the obvious embedding
if $\F=\R$, the map $w\mapsto (\oline w,w)$ if $\F=\C$, and the map $w\mapsto\breve w$ if $\F=\qH$. 
\end{rmk}
The action of $G_r$ on $A$ extends uniquely to a holomorphic action of $\bfG_r$ on $A_\C$. 
Similarly, the action of $G_n\times G_r$ on $W$ extends uniquely to a holomorphic action of $\bfG_n\times \bfG_r$ on $W_\C$. These holomorphic actions are explicitly  described in Table 4 below.
\vspace{1mm}
\begin{center}
{\
\begin{tabular}{cccc}
\bottomrule
$\F$ & $\bfG_r\Gact A_\C$ & 
$(\bfG_n\times \bfG_r)\Gact W_\C$ \\
\toprule
$\R$ & $\sS^2(\C^r)^*$ & 
$\C^n\otimes(\C^r)^*$ \\
$\C$ & $\C^r\otimes (\C^r)^*$ & 
$\left((\C^n)^*\otimes \C^r\right)\oplus \left(\C^n\otimes (\C^r)^*\right)$\\
$\qH$ & $\Lambda^2(\C^{2r})^*$ & 
$\C^{2n}\otimes (\C^{2r})^*$\\
\bottomrule

\end{tabular}
}
\vspace{2mm}

Table 4.
\end{center}

\begin{dfn}
For every integer $m\geq 1$, let $\mathrm H_m\subset\mathrm{GL}_m(\C)$ denote the standard Cartan subgroup of diagonal matrices, and let   $\mathrm B_m\subset\mathrm{GL}_m(\C)$ denote the  standard  Borel subgroup of upper triangular matrices. Let 
$\eps_1,\ldots,\eps_m$ denote the standard
generators of the weight lattice of $\mathrm{GL}_m(\C)$. For every
$\lambda:=(\lambda_1,\ldots,\lambda_m)\in\Z^m$ such that $\lambda_1\geq\cdots\geq\lambda_m$, we denote the $\mathrm{GL}_m(\C)$-module with $\mathrm{B}_m$-highest weight $\sum_{i=1}^m\lambda_i\eps_i$ by $M_\lambda$.
\end{dfn}

\begin{dfn}
The standard  Borel subgroup
of $\bfG_m$ will be denoted by $\bfB_m$. 
In cases $\F=\R$, $\F=\C$, and $\F=\qH$, the group 
$\bfB_m\subset \bfG_m$ equals 
$\mathrm B_m$, $\mathrm B_m\times\mathrm B_m$, and $\mathrm B_{2m}$. The standard Cartan subgroup of $\bfG_m$ will be denoted by $\bfH_m$.
\end{dfn}

\begin{rmk}
\label{rmkEll+sseq}
For every integer $m\geq 1$, we define $\cP_m:=\{(\lambda_1,\ldots,\lambda_m)\in\Z^m\,:\,\lambda_1\geq \cdots\geq \lambda_m\geq 0\}$. From now on, we will denote the length (that is, the number of nonzero parts) of a partition $\lambda\in\cP_m$ by  $\ell(\lambda)$. For two partitions $\lambda\in\cP_m$ and $\mu\in\cP_{k}$, where $k,m\geq 1$, we write $\lambda\sseq \mu$ if and only if $\ell(\lambda)\leq \ell(\mu)$ and $\lambda_i\leq \mu_i$ for every $1\leq i\leq \ell(\lambda)$.
\end{rmk}

Let $\sP(A)$ and $\sP(W)$ denote the $\C$-algebras of polynomials on $A_\C$ and $W_\C$. The canonical $G_r$-action on  $\sP(A)$, given by $g\cdot f(a):=f(g^{-1}\cdot a)$ for $g\in G_r$, $f\in\sP(A)$, and $a\in A$, extends uniquely to a holomorphic $\bfG_r$-action on $\sP(A)$. Similarly,  the canonical $G_n\times G_r$-action on  $\sP(W)$ 
 extends uniquely to a holomorphic $\bfG_n\times \bfG_r$-action on $\sP(W)$.
 The pullback
\begin{equation}
\label{q*PAPW}
\mapvWA^*:\sP(A)\to \sP(W)\ ,\ f\mapsto f\circ \mapvWA
\end{equation}
 is  a $\bfG_r$-equivariant  embedding of $\C$-algebras. 
The image of $\mapvWA^*$ is precisely described  
by the First Fundamental Theorem of invariant theory
\cite[Sec. 5.2.1]{GW09}. In particular,
$\mapvWA^*(\sP(A))=\sP(W)^{\bfK_n}$. 

By classical invariant theory
(for example see \cite{GW09}), $\sP(A)$  decomposes into a direct sum of irreducible $\bfG_r$-modules which are naturally parametrized by partitions $\lambda\in\cP_r$. Thus,
\begin{equation}
\label{P(A)=}
\sP(A)\simeq\bigoplus_{\lambda\in \cP_r}F_\lambda,
\end{equation}
where $F_\lambda$ is the irreducible $\bfG_r$-module
corresponding to $\lambda\in\cP_r$. 
In fact   $F_\lambda\simeq M_\lambda^*\otimes M_\lambda^{}$ if $\F=\C$, 
and $F_\lambda\simeq M_{\lambul}$ if $\F=\R,\qH$, 
where 
$\lambul:=(2\lambda_1,\ldots,2\lambda_r)\in\cP_r$ if $\F=\R$, and 
$\lambul:=(\lambda_1,\lambda_1,\ldots,\lambda_r,\lambda_r)\in\cP_{2r}$ if $\F=\qH$.

The map \eqref{q*PAPW} is 
$\bfG_r$-equivariant, and therefore $F_\lambda$ occurs as a $\bfG_r$-submodule of $\sP(W)$ for every $\lambda\in\cP_r$. 
Therefore by the well known 
$(\mathrm{GL}_n,\mathrm{GL}_r)$ duality (see \cite[Sec. 2.1]{HoweSchur} or \cite[Sec. 5.6.2]{GW09}),  for every $\lambda\in\cP_r$ there exists a unique irreducible $\bfG_n$-module $E_\lambda$ such that $E_\lambda\otimes F_\lambda$ occurs in $\sP(W)$ as a $\bfG_n\times \bfG_r$-submodule.
\begin{rmk} 
\label{rmk-Sphe}
Let $\sP_{\bfK_n}(W)$ denote the direct sum of irreducible $\bfK_n$-spherical $\bfG_n$-submodules of  $\sP(W)$. Then indeed 
\[
\sP_{\bfK_n}(W)\simeq\bigoplus_{\lambda\in\cP_r}E_\lambda\otimes F_\lambda\ \text{ as }
\bfG_n\times\bfG_r\text{-modules.}
\] 
\end{rmk}

\begin{rmk}
\label{rmk-2.2}
For every integer $l\geq  1$, let $S_{l\times l}$ be the $l\times l$ matrix with 1's in $(i,l-i+1)$-entry for every $1\leq i\leq l$, and with 0's elsewhere. 
Consider $g_\circ\in \bfG_n$ defined as follows. If $\F=\R$, then we set
\[
g_\circ:=\begin{bmatrix}
\frac{1+i}{2} I_{l\times l}& \frac{1-i}{2} S_{l\times l}\\
\frac{1-i}{2} S_{l\times l}& \frac{1+i}{2} I_{l\times l}
\end{bmatrix}
\text{ for }n=2l,
\text{ and }
g_\circ:=\begin{bmatrix}
\frac{1+i}{2} 
I_{l\times l}& 0_{l\times 1} & \frac{1-i}{2} S_{l\times l}\\
 0_{1\times l} & 1 &  0_{1\times l}\\
\frac{1-i}{2} S_{l\times l}&  0_{l\times 1} & \frac{1+i}{2} I_{l\times l}
\end{bmatrix}
\text{ for }n=2l+1.
\]
If $\F=\C$, then we set $g_\circ $ equal to the identity element of $\bfG_n$. Finally, if $\F=\qH$, then we set 
\[
g_\circ:=\begin{bmatrix}
 S_{n\times n}& 0\\
 0& I_{n\times n} 
\end{bmatrix}.
\]
We remark that when $\F=\R,\qH$, the map $\bfK_n\to \bfG_n$, $g\mapsto g_\circ g g_\circ^{-1}$ 
is  the  embedding $\mathrm{O}_n(\C)\into
\mathrm{GL}_n(\C)$ 
or 
$\mathrm{Sp}_{2n}(\C)\into\mathrm{GL}_{2n}(\C)$
that is given in \cite[Sec. 2.1.2]{GW09}.
In particular, 
in all cases $(g_\circ^{-1}\bfH_ng_\circ)\cap \bfK_n$ and 
$(g_\circ^{-1}\bfB_n g_\circ)\cap \bfK_n$  are Cartan and Borel subalgebras of $\bfK_n$. 

Set $\bfK:=\bfK_n$, and let $\bfM\subset\bfK$ denote the complexification of $M$. The $M$-spherical $K$-module $V_\mu$ is  naturally also an $\bfM$-spherical $\bfK$-module. By comparing  the calculation of highest weights of $\bfM$-spherical $\bfK$-modules in \cite[Sec. 12.3.2]{GW09} (the pertinent cases are types BDI, AIII, and CII) with the parametrization of representations of $\bfK$ by partitions that  uses generalized Schur--Weyl duality (see \cite[Thm 10.2.9]{GW09} and \cite[Thm 10.2.12]{GW09}), it follows that for every $\mu\in\cP_r$, the module $V_\mu$ is isomorphic to the $\bfK$-submodule of 
$E_\mu^*$ generated by $g_\circ^{-1}\cdot v_\lambda$, where $v_\lambda$ denotes the  $\bfB_n$-highest weight of $E_\mu^*$.
\end{rmk}

\begin{dfn}
Let $G$ be a group, and let $E$ and $F$ be $G$-modules. We 
set \[
[E:F]_{G}:=
\dim\Hom_{G}^{}(E,F).
\]
\end{dfn}

\begin{rmk}
\label{RmkLiLRChd}
In the following, we will need  the Littlewood--Richardson Rule, which we now recall (for a more elaborate reference,  see for example \cite[Sec. I.9]{Macdonald}). For a semistandard skew tableau $T$, the word $w(T)$ corresponding to $T$ is defined as the sequence of integers obtained by reading the contents of boxes of $T$ from right to left and from top to bottom. A word $w_1\cdots w_k$ in letters $\{1,\ldots,N\}$ is called a lattice permutation word if for every $1\leq i\leq k$ and every $1\leq j\leq N-1$, the number of occurrences of $j$ in $w_1\ldots w_i$ is greater than or equal to the number of occurrences of $j+1$. 
Now let $\lambda,\mu,\nu\in\cP_m$, where $m\geq 1$. The Littlewood--Richardson Rule states that 
$[M_\nu:M_\lambda\otimes M_\mu]_{\mathrm{GL}_m(\C)}^{}$  is 
equal to the number of tableaux $T$ of shape $\nu\bls \mu$ and weight $\lambda$ such that $w(T)$ is a lattice permutation word. In particular, if 
$[M_\nu:M_\lambda\otimes M_\mu]_{\mathrm{GL}_m(\C)}^{}\neq 0$, then $\mu,\lambda\sseq \nu$.

\end{rmk}

\begin{lem}
\label{lem-crtaag}
Let $\lambda,\mu\in\cP_r$. Then
$[V_\mu:E_\lambda]_{\bfK_n}^{}=
[V_\mu:E_\lambda^*]_{\bfK_n}^{}$.
\end{lem}
\begin{proof}
It is enough to show that $E_\lambda\simeq E_\lambda^*$  as $\bfK_n$-modules.
Let $\theta_n:\bfG_n\to\bfG_n$ denote the automorphism of $\bfG_n$ that is obtained by holomorphic extension of the 
Cartan involution 
$g\mapsto (g^{\dagger})^{-1}$
of $G_n$.
Let $E_\lambda^{\theta_n}$ be the $\bfG_n$-module that results from twisting $E_\lambda$ by $\theta_n$.
Then $E_\lambda^{\theta_n}\simeq
E_\lambda^*$ 
as  $\bfG_n$-modules. Moreover, since $\theta_n$ fixes $\bfK_n$ pointwise, $E_\lambda\simeq E_\lambda^{\theta_n}$ as $\bfK_n$-modules. The $\bfK_n$-module isomorphism $E_\lambda\simeq E_\lambda^*$ now follows immediately. 
\end{proof}

When $\F=\C$, 
branching from $\bfG_n$ to $\bfK_n$ is described by the Littlewood--Richardson Rule. 
 The next proposition is a branching from $\bfG_n$ to $\bfK_n$ when $\F\neq \C$. 
\begin{prp}
\label{thmLIL}
Assume that $\F=\R$ or $\F=\qH$, and  let $\lambda,\mu\in \cP_r$.
Then 
\begin{equation}
\label{VmuEla]}
[V_\mu:E_\lambda]_{\bfK_n}^{}
=
\sum_{\xi\in\cP_r}
[E_\lambda:
E_\mu\otimes E_{\xi}]_{\bfG_n}^{}.
\end{equation}
\end{prp}
\begin{proof}
The statement follows as a special case of the Littlewood Restriction Theorem, which  was first proved in 
\cite{Littlewood} (see also \cite{Koiketerada}, \cite{Gavarini}, and \cite[Sec. 1.3]{HTW}). We now outline the calculations that are needed to deduce the proposition from the Littlewood Restriction Theorem. We remark that the $\bfG_n$-modules that appear in the statement of the Littlewood Restriction Theorem are polynomial representations, whereas the $E_\lambda$ are indeed contragredients of polynomial representations. To get around this issue, we note that we can replace the left hand side of 
\eqref{VmuEla]} by $[V_\mu,E_\lambda^*]_{\bfK_n}^{}$ (see Lemma \ref{lem-crtaag})
and the terms  of the right hand side by
$[E^*_\lambda:
E^*_\mu\otimes E^*_{\xi}]_{\bfG_n}^{}
$.
Using Remark \ref{rmk-2.2} we can write the highest weight of $V_\mu$ explicitly, and then we can determine the partition that corresponds to $V_\mu$ in the parametrization of \cite[Thm 10.2.9]{GW09} and 
\cite[Thm 10.2.12]{GW09}.
If 
$[E_\lambda:E_\mu\otimes M_\xi]_{\bfG_n}^{}\neq 0$ for some $\xi\in\cP_n$, then Remark \ref{RmkLiLRChd} implies that $\ell(\xi)\leq r$ when $\F=\R$, and $\ell(\xi)\leq 2r$ when $\F=\qH$.  
 By putting all of these facts together, we can verify   that the statement of the proposition is a special case of the Littlewood Restriction Theorem.
\end{proof}

\section{The quadratic Capelli operators}
\label{Sec-theOprDlam}

In this section we define the differential operators $D_{\lambda,s}$. 
Let $\sD(A):=\bigoplus_{m=0}^\infty\sD^m(A)$
denote the $\C$-algebra of constant coefficient differential
operators on $A$, endowed with the usual $\Z$-grading.
We define $\sD(W):=\bigoplus_{m=0}^\infty\sD^m(W)$ similarly. 
There are natural $G_r$-actions on $\sD(A)$ and $\sD(W)$, and as in
Section \ref{Sec-PaamPart},
 these actions extend uniquely to holomorphic $\bfG_r$-actions on the same vector spaces.
Furthermore, the  canonical isomorphisms 
\begin{equation}
\label{sDMM}
\sD^m(A)\simeq \sP^m(A)^*\text{ and }\sD^m(W)\simeq \sP^m(W)^*
\end{equation} 
are $\bfG_r$-equivariant. 

Let $\sPD(A)$ and $\sPD(W)$ denote the algebras of polynomial coefficient differential operators on $A$ and $W$. The multiplication map  results in  isomorphisms of vector spaces
\begin{equation}
\label{sPDPDWA}
\sP(A)\otimes\sD(A)\simeq\sPD(A)\,\text{ and }\, \sP(W)\otimes\sD(W)\simeq\sPD(W).
\end{equation}

From now on, we set \[
\boldth:G_r\to G_r
\ ,\ 
\boldth(g):=(g^\dagger)^{-1}.
\]
\begin{dfn}
We define  bilinear forms 
$(\cdot,\cdot)_W^{}:W\times W\to\R $
and
$(\cdot,\cdot)_A^{}:A\times A\to\R $
by
\[
(x,y)_W^{}:= 
\Re(\tr(x^\dagger y))
\,\text{ and }\,
(x,y)^{}_A:=\Re(\tr(xy)).
\]
\end{dfn}

The bilinear form
$
(\cdot,\cdot)_W^{}
$
is  $K_n$-invariant and  
$\boldth$-invariant, that is,
\[
(kx,ky)_W^{}=(x,y)_W^{}\,\text{  and }\,(g\cdot x,y)_W^{}=(x,\boldth(g)^{-1}\cdot y)_W^{}
\,\text{ for }\,
x,y\in W,\  
k\in K_n,\, \text{ and }\, g\in G_r.
\]
Similarly, the bilinear form 
$
(\cdot,\cdot)_A^{} 
$
is $\boldth$-invariant, that is, $(g\cdot x,y)_A^{}=(x,\boldth(g)^{-1}\cdot y)_A^{}$ for $x,y\in A$ and $g\in G_r$. 
The bilinear forms $(\cdot,\cdot)_W^{}$ and $(\cdot,\cdot)_A^{}$ yield canonical isomorphisms 
\begin{equation}
\label{iotaAWs}
\iota_W^{}:W\to W^*\text{ and }
\iota_A^{}:A\to A^*.
\end{equation} These maps extend to 
$\bfG_r$-equivariant isomorphisms of $\C$-algebras 
\begin{equation}
\label{PDWPDA}
\iota_W^{}:\sD(W)\simeq \sP(W)\,\text{ and }\,
\iota_A^{}:\sD(A)\simeq\sP(A).
\end{equation}
Since $\iota_W^{}$ is also $K_n$-equivariant, it restricts to an isomorphism
$\sP(W)^{K_n}\to \sD(W)^{K_n}$. Consequently, we obtain a $\bfG_r$-equivariant isomorphism of $\C$-algebras
\[
(\iota_W^{})^{-1}\circ\mapvWA^*\circ\iota_A^{}:\sD(A)\to \sD(W)^{K_n}.
\]
Set $\breve\psi:=(\iota_W^{})^{-1}\circ\mapvWA^*
\circ\iota_A^{}$. From \eqref{sPDPDWA} it follows that
the map
\begin{equation}
\label{iotasPDAW}
\iota:\sPD(A)\into \sPD(W)^{K_n},\ 
\iota:=
\breve\psi\otimes \mapvWA^*
\end{equation}
is an  embedding of $\bfG_r$-modules.
From  \eqref{P(A)=} and \eqref{sDMM}  it follows that
$\sD(A)\simeq\bigoplus_{\lambda\in\cP_r} F_\lambda^*$, so that
\[
\sPD(A)\simeq \sP(A)\otimes \sD(A)\simeq
\bigoplus_{\lambda,\mu\in\cP_r}
F_\lambda^{}\otimes F_\mu^*\simeq
\bigoplus_{\lambda,\mu\in\cP_r}\Hom_\C(F_\mu,F_\lambda).
\]
By Schur's Lemma $[F_\mu:F_\lambda]_{\bfG_r}^{}\leq 1$, and equality
occurs if and only if $\lambda=\mu$. Thus,
\begin{equation}
\label{EqbfDl}
\sPD(A)^{\bfG_r}\simeq 
\bigoplus_{\lambda,\mu\in\cP_r}\Hom_{\bfG_r}^{}(F_\mu,F_\lambda)\simeq
\bigoplus_{\lambda\in\cP_r}\C I_\lambda,
\end{equation}
where $I_\lambda$ denotes the identity element of $\Hom_\C(F_\lambda,F_\lambda)$.
Let $C_\lambda\in\sPD(A)^{\bfG_r}$
be the differential operator that corresponds to $ I_\lambda$ by the isomorphism \eqref{EqbfDl}.
We now set
\begin{equation}
\label{DtildeLam}
\widetilde{D}_\lambda:=\iota(C_\lambda)\in\sPD(W)^{K_n\times G_r},
\end{equation}
where $\iota$ is the map defined in \eqref{iotasPDAW}. 

Set $\Psi(x):=\det(x^\dagger x)$ for $x\in W$
(in the case $\F=\qH$ we define
$\det(z):=\det(\breve z)$ for 
$z\in \mathrm{Mat}_{r\times r}(\qH)$, where $\breve z$ is as in Remark \ref{rmk-qHin2C}).
For every $\lambda\in\cP_r$ and every $s\in\C$, let $\widetilde{D}_{\lambda,s}$ be the differential operator on $X$ defined by
\[
\widetilde D_{\lambda,s}:=\Psi^s\widetilde D_\lambda \Psi^{-s}.
\]
\begin{dfn}
\label{dfnofDllam}
For $\lambda\in\cP_r$ and $s\in\C$, we set 
$(D_{\lambda,s} f)(\varphi(x)):=\widetilde{D}_{\lambda,s}(f\circ \varphi)(x)$ for every $f\in C^\infty(Y)$ and every $x\in X$,
where $X\overset{\varphi}{\longrightarrow}Y$ is the 
map
defined in Section \ref{sec:introduction}.
\end{dfn}
Since ${D}_{\lambda,s}$ does not increase supports, by Peetre's Theorem \cite[Thm II.1.4]{Helg}
it  is a  differential  operator 
on $Y$. From $K_n$-invariance of 
$\widetilde{D}_{\lambda,s}$ it follows that $D_{\lambda,s}$ is also
$K_n$-invariant.

\section{The polynomials $P_\lambda(x;\tau,\alpha)$}
\label{sec:van-uniqn}
In this section we review the definition and properties of 
the polynomials $P_\lambda(x;\tau,\alpha)$.
\begin{dfn}
Let $\lambda:=(\lambda_1,\ldots,\lambda_r)\in\Z^r$. 
A Laurent polynomial $f(x_1,\ldots,x_r)$ in variables $x_1,\ldots,x_r$ is called $\lambda$-monic if the coefficient of $x_1^{\lambda_1}\cdots x_r^{\lambda_r}$ in $f(x_1,\ldots,x_r)$ is equal to 1.
\end{dfn}
Recall that 
the Weyl group $\mathsf W$ of type $BC_r$ is a semidirect product $\mathsf W:=S_r\ltimes \{\pm 1\}^r$, where $S_r$ denotes the symmetric group on $r$ letters. 
In \cite{Okounkov}, Okounkov 
defined a family of  Laurent polynomials \[
\POk_\lambda(x;q,t,a)\in\C(q,t,a)[x_1^{\pm 1},\ldots,x_r^{\pm1}],
\] parametrized by partitions $\lambda\in\cP_r$ (we use the notation of \cite[Sec. 5]{Koornwinder}).
 Every $\POk_\lambda$ is the unique $\lambda$-monic Laurent polynomial
of degree $|\lambda|$ that is
invariant under the action of 
$\mathsf W$ on the $x_i$'s by 
permutations and inversions, 
and satisfies the vanishing condition
\[
\POk_\lambda(aq^\mu t^\delta;q,t,a)=0\text{ unless } \lambda\subseteq \mu,
\]
where 
$\delta:=(r-1,\ldots,0)$ and 
$\mu\in\cP_r$  (see \cite[Sec. 5.3]{Koornwinder}).
Here as usual we define $q^\mu t^\delta:=(q^{\mu_1}t^{\delta_1},\ldots,q^{\mu_r}t^{\delta_r})$. 
The
polynomials 
$\POk_\lambda(x;q,t,a)$
are analogues (for the $BC_r$-type root system) of  
the $q$-deformed interpolation Macdonald polynomials
defined by Knop \cite{Knop}
and Sahi \cite{SahiInt}.

By taking  the $q\to 1$ limit of
$\POk_\lambda$
 (see \cite[Def. 7.1]{Koornwinder}), one obtains a  polynomial
$
P_\lambda(x;\tau,\alpha)\in\C(\tau,\alpha)[x_1,\ldots,x_r]
$.
More precisely, 
\[
P_\lambda(x;\tau,\alpha)
:=\lim_{q\uparrow 1}(1-q)^{-2|\lambda|}\POk_\lambda
(q^x; q,q^\tau,q^\alpha),
\]
where $|\lambda|:=\sum_i\lambda_i$.
From the symmetry property of $\POk_\lambda$ it follows that $P_\lambda(x;\tau,\alpha)$ is 
invariant under permutations and sign changes of $x_1,\ldots,x_r$.
\begin{dfn}
A polynomial in variables $x_1,\ldots,x_r$ is called \emph{even-symmetric} if it is invariant under  permutations and sign changes of the $x_i$'s.
\end{dfn}

A combinatorial formula for $P_\lambda(x;\tau,\alpha)$ is given in \cite[Sec. 7]{Koornwinder}. To recall this formula, we need some terminology.  Every  partition $\lambda$ can be represented by a Young diagram consisting of boxes $\flat:=\flat(i,j)$, where \[(i,j)\in\{(p,q)\in\Z^2\ :\ 1\leq p\leq \ell(\lambda)\text{ and }
1\leq q\leq \lambda_i\}.
\] 
The
arm length and leg length of a box $\flat:=\flat(i,j)$
in the Young diagram of $\lambda$  are  $a_\lambda(\flat):=\lambda_i-j$ and 
$l_\lambda(\flat):=|\{k>i\,:\,\lambda_k\geq j\}|$. 
We also set
$a_\lambda'(\flat):=j-1$ (arm co-length) and $l_\lambda'(\flat):=i-1$ (leg co-length). 
By a reverse tableau of shape $\lambda$ with entries in $\{1,\ldots,r\}$ we mean a filling of  the Young diagram that corresponds to $\lambda$, with  weakly decreasing rows and strongly decreasing columns.
For a reverse tableau $T$ of shape $\lambda$  and an integer $k\in\{0,\ldots,r\}$, 
let $\lambda^{(k)}\sseq \lambda$ 
be the partition corresponding to the boxes $\flat\in \lambda$ that satisfy $T(\flat)>k$. Thus for $1\leq k\leq n$, $\lambda^{(k-1)}\bls \lambda^{(k)}$ is the horizontal strip consisting  of the boxes that contain $k$. Finally, for two partitions $\nu\sseq\mu$, we define 
$(R\bls C)_{\mu\bls \nu}$
 to be the set of boxes which are in a row of $\mu$ intersecting with $\mu\bls
\nu$, but not in a column of $\mu$ intersecting with $\mu\bls
\nu$.
Set
\[
b_{\mu}(\flat;\tau):=\frac{
a_{\mu}(\flat)+\tau(l_{\mu}(\flat)+1)
}{
a_{\mu}(\flat)+\tau l_{\mu}(\flat)+1
}\,
\text{ and }
\,\psi_T(\tau):=
\prod_{i=1}^r\ 
\prod_{\flat\in(R\bls C)_{\lambda^{(i-1)}\bls \lambda^{(i)}}}
\frac{b_{\lambda^{(i)}}(\flat;\tau)}{b_{\lambda^{(i-1)}}(\flat;\tau)}.\]
Then 
\begin{equation}
\label{formcmbPip}
P_\lambda(x;\tau,\alpha)
=
\sum_T
\psi_T(\tau)\prod_{\flat\in\lambda}
\left(
x_{T(\flat)}^2-(a'_\lambda(\flat)+\tau(n-T(\flat)-l'_\lambda(\flat))+\alpha)^2
\right),
\end{equation}
where the sum is over all reverse tableaux $T$ of shape $\lambda$ 
with entries in $\{1,\ldots,r\}$.
\begin{prp}
\label{prpuniqPip}
Fix real numbers $\alpha,\tau>0$. For every partition $\lambda\in\cP_r$,  the polynomial \[
Q_\lambda^{}(x):=P_\lambda(x;\tau,\alpha)\in
\C[x_1,\ldots,x_r]
\]
 is the unique $2\lambda$-monic, even-symmetric polynomial that 
satisfies the vanishing condition
\begin{equation}
\label{vancond}
Q_\lambda(\mu+\varrho_{\tau,\alpha})=0
\text{ if }|\mu|\leq |\lambda|\text{ and }\mu\neq \lambda,
\end{equation}
where $\varrho_{\tau,\alpha}$ is given in  \eqref{rhotaualph}.
\end{prp}
\begin{proof}
The existence statement follows from 
\cite[Sec. 7]{Koornwinder}
and the fact that 
 the specialization of 
$P_\lambda$ at the  values of $\tau$ and $\alpha$ is well-defined, because
when $\tau>0$, the denominators of the coefficients $\psi_T(\tau)$ of  $P_\lambda(x;\tau,\alpha)$ that appear in the combinatorial formula  
\eqref{formcmbPip} 
do not vanish. For the uniqueness statement, we use a method based on \cite{Sahi}. 
First note that from $\alpha,\tau>0$ and \cite[Eq. (7.5)]{Koornwinder} it follows that $Q_\lambda
(\lambda+\varrho_{\tau,\alpha})\neq 0$.
Next fix an integer $N>0$,  set $\cI_{N}:=\{\mu\in\cP_r\,:\,|\mu|\leq N\}$, and let $\cS_N$ denote the vector space of even-symmetric polynomials  in $x_1,\ldots,x_r$  of degree at most $2N$.   Note that $\dim\cS_N=|\cI_N|$. For every $\mu\in\cI_N$, we consider the linear maps
\[
L_\mu:\cS_N\to\C\ ,\ f\mapsto f(\mu+\varrho_{\tau,\alpha}).
\]
Next we define a total order  $\prec$ on $\cI_N$, as follows.
We set $\mu\prec \nu$ for every $\mu,\nu\in\cI_N$ which satisfy 
$|\mu|<|\nu|$,
and then we extend the resulting partial order to a total order on $\cI_N$.
Then the matrix \[
\left[L_{\mu}(Q_{\mu'})
\right]_{\mu,\mu'\in\cI_N}
\] 
is upper triangular with 
nonzero entries on the diagonal. 
It follows that the linear map 
\[
\cS_N\to\C^{\dim\cS_N}\ ,\ f\mapsto \big[f(\mu+\varrho_{\tau,\alpha})\big]_{\mu\in \cI_N}.
\]
is invertible. Uniqueness of $Q_\lambda$ follows immediately from the latter statement. 
\end{proof}

Recall that $d:=\dim(\F)$. For every $\lambda\in\cP_r$, we set \begin{equation}
\label{defConstgamla}
\gamma_\lambda:=\frac{(-2d)^{|\lambda|}}{\prod_{\flat\in\lambda}
\left(
\frac{d}{2}(a_\lambda(\flat)+1)
+l_\lambda(\flat)
\right)}.
\end{equation}
This is the scalar that appears in the statement of Theorem \ref{maintheorem}. 
\section{The vanishing property of $c_{\lambda,s}(\mu)$}
\label{sec:vanishprpp}

In this section, we prove a few  technical statements which will be used in the proof of  Theorem \ref{maintheorem}. The ultimate goal of this section is to prove Proposition \ref{PrpVanCondd}.
The proof of Theorem 
\ref{maintheorem}
will be completed in 
Section \ref{Sec-pf-Main-Thm}.
Recall that $\varrho:=(\varrho_1,\ldots,\varrho_r)$ is the vector defined in \eqref{formulRho}.
\begin{lem}
\label{clmusdd+r}
Let $\lambda\in\cP_r$. Then there exists a 
 polynomial $d_\lambda(x,s)$ which  
is even-symmetric
in $x:=(x_1,\ldots,x_r)$,
has $x$-degree and $s$-degree at most $2|\lambda|$, and 
satisfies 
\[
c_{\lambda,s}(\mu)=d_\lambda(\mu+\varrho,s),
\]
 for every $s\in\C$ and every 
$\mu\in\cP_r$.
\end{lem}
\begin{proof}
The proof is similar to \cite[Lemma 3.1]{SahiZhang}.
Throughout the proof we fix $\lambda$. Recall that $K:=K_n$ and $M:=K_r\times K_{n-r}$. 
For every $s\in\C$, 
the differential operator
$D_{\lambda,s}$
is $K$-invariant and has order at most $2|\lambda|$. 
By \eqref{eqwtiMUU}, for every $\mu\in\cP_r$, the highest weight $\widetilde\mu$ of $V_\mu$ satisfies $\widetilde\mu\big|_{\g a_\C}=\sum_{i=1}^r2\mu_i\sfe_i$. 
Therefore from the Harish-Chandra homomorphism \cite[Chap. II]{Helg} it follows that for every $s\in\C$, the scalar $c_{\lambda,s}(\mu)$ is an  even-symmetric polynomial 
of degree at most $2|\lambda|$
evaluated at $\mu+\varrho$. 
Next we set
\[
x_\circ:=\begin{bmatrix}I_{r\times r}\\ 0\end{bmatrix}\in X,
\]
and we denote the image of $x_\circ$ in $Y$ by $y_\circ$.
Choose an $M$-fixed vector $h_\mu\in V_\mu\subseteq C^\infty(Y)$ such that $h_\mu(y_\circ)=1$. 
By evaluating both sides of  the relation $D_{\lambda,s}h_\mu=c_{\lambda,s}(\mu)h_\mu$ at $y_\circ$, we obtain
\begin{equation}
\label{clmustd}
c_{\lambda,s}(\mu)=\widetilde{D}_\lambda(\det(x^\dagger x)^{-s}\big(h_\mu\circ\varphi)\big)(x_\circ),
\end{equation}
where $\widetilde{D}_\lambda$ is defined in \eqref{DtildeLam}. Since $\widetilde D_\lambda$ is a polynomial coefficient differential operator of order at most 
$2|\lambda|$,  from the Leibniz rule it follows that for fixed $\mu$, the right hand side of 
\eqref{clmustd} is a polynomial in $s$ of degree at most $2|\lambda|$. Consequently, 
\[
c_{\lambda,s}(\mu)=\sum_{j=0}^{2|\lambda|}a_j(\mu)s^j\text{ for every }s\in\C.
\]
For $2|\lambda|+1$ distinct values of $s$, we obtain a linear system in the coefficients $a_j(\mu)$ whose coefficients form an invertible Vandermonde matrix. Since
fox fixed $s\in\C$ we have shown that $c_{s,\lambda}(\mu)$ is a polynomial in $\mu$ of degree at most $2|\lambda|$, it follows that $a_j(\mu)$ is also a polynomial in $\mu$ of degree at most $2|\lambda|$. 
Consequently, both the $x$-degree and the $s$-degree of $c_{s,\lambda}$ are at most $2|\lambda|$.
The statement that 
$d_\lambda$ is even-symmetric follows from the fact that
$c_{\lambda,s}(\mu)$ is even-symmetric as a polynomial in $\mu+\varrho$.
\end{proof}
\begin{dfn}
\label{DFNPCHI}
For every $\chi\in\cP_r$, let $\sP_\chi(W)$ denote the $F_\chi$-isotypic component of $\sP(W)$.
\end{dfn}
We remark that $\sP_\chi(W)$ is $\bfG_n$-invariant.
\begin{prp}
\label{VetaPm+l*}
Let $\lambda,\nu\in \cP_r$, and let $m\in\Z$  such that $m\geq \max\{\lambda_1,\nu_1\}$. Set \[
\eta:=(m-\nu_r,\ldots,m-\nu_1)
\text{ and }\chi:=(m-\lambda_r,\ldots,m-\lambda_1).
\]
If $[V_\eta:\sP_{\chi}(W)]_{\bfK_n}^{}>0$, then $\lambda\sseq \nu$.

\end{prp}
\begin{proof}
Recall from Section \ref{Sec-PaamPart} that 
by $(\mathrm{GL}_n,\mathrm{GL}_r)$-duality,
$\sP_{\chi}(W)\simeq E_\chi^{\oplus\dim F_\chi}$ as $\bfG_n$-modules.
Therefore we can assume that $[V_\eta:E_{\chi}]_{\bfK_n}^{}> 0$.
We will consider  two separate cases.

\textbf{Case I:} $\F=\R,\qH$. 
Proposition
\ref{thmLIL}
implies that $
[E_{\chi}:E_{\eta}
\otimes E_{\xi}]_{\bfG_n}^{}>0
$
 for some $\xi\in\cP_r$,
and Remark \ref{RmkLiLRChd} implies that
$\eta\subseteq \chi$, hence $\lambda\sseq\nu$.

\textbf{Case II:} $\F=\C$. Let $\widetilde\chi\in\cP_n$ be obtained from $\chi$ by adding $n-r$ zeros on the right of the parts of $\chi$. 
Our assumption entails that the restriction of
the $\mathrm{GL}_n(\C)\times \mathrm{GL}_n(\C)$-module $M_{\widetilde \chi}^{}\otimes M_{\widetilde \chi}^{*}$
 to the diagonal subgroup $\mathrm{GL}_n(\C)$ contains the $\mathrm{GL}_n(\C)$-module $M_{\widetilde{\eta}}$, where \[
\widetilde{\eta}:=(m-\nu_r,\ldots,m-\nu_1,0,\ldots,0,-m+\nu_1,\ldots,-m+\nu_r)\in\cP_n.
 \]
Remark \ref{RmkLiLRChd}
implies  that $-m+\nu_i\geq -m+\lambda_i$ for every $1\leq i\leq r$, hence $\lambda\subseteq \nu$.
\end{proof}

%
The proof of our next result, Proposition \ref{propR^mKdec},  
is based on facts from the theory of symmetric functions, which we now review quickly  
(for a comprehensive reference, see \cite{Macdonald}). 
Let \[
\cP:=\varinjlim_k\cP_k
\] be the set of all partitions, where the maps $\cP_k\to\cP_{k+1}$ are given by $(\lambda_1,\ldots,\lambda_k)\mapsto (\lambda_1,\ldots,\lambda_k,0)$, and let 
\[
\Lambda:=\varprojlim \C[x_1,\ldots,x_k]^{S_k}
\]
 denote the ring of symmetric functions, where the maps
$\C[x_1,\ldots,x_{k+1}]\to \C[x_1,\ldots,x_{k}]$ are given by $f\mapsto f(x_1,\ldots,x_k,0)$. 
As usual, we equip $\Lambda$ with a scalar product defined by 
$
\lag h_\lambda,m_\mu\rag_\Lambda^{}:=\delta_{\lambda,\mu}
$,
where $h_\lambda$ and $m_\mu$ are the complete and monomial symmetric functions associated to  $\lambda,\mu\in\cP$. The Schur functions $s_\lambda$, $\lambda\in\cP$, form an orthonormal basis for $\Lambda$. For every two $\lambda,\mu\in\cP$ such that $\mu\subseteq\lambda$, the  skew Schur function $s_{\lambda\bls\mu}\in\Lambda$ satisfies the relation
$
\lag s_{\lambda\bls\mu},s_\nu\rag_\Lambda^{}=
\lag s_\lambda,s_\mu s_\nu\rag_\Lambda^{}
$
 for every  $\nu\in\cP$.
It is well known \cite[Ex. 7.56(a)]{Stanley2} that for any skew diagram $\lambda\bls \mu$ we have 
\begin{equation}
\label{180deg}
s_{\lambda\bls \mu}=s_{(\lambda\bls\mu)^\circ},
\end{equation}
where $(\lambda\bls\mu)^\circ$ denotes the skew diagram obtained by a 180 degree rotation of $\lambda\bls\mu$.

\begin{prp}
\label{propR^mKdec}
Let $m\geq 1$ be an integer, let
$\mathbf m:=(m,\ldots,m)\in\cP_r$ be the partition corresponding to the $r\times m$ rectangular Young diagram, 
and set $\cP(\mathbf m):=\{\lambda\in\cP_r\,:\,\lambda\sseq\mathbf m\}$.
Let 
$\sP_\mathbf m(W)$ be as in Definition \ref{DFNPCHI}.
Then as $\bfK_n$-modules, 
\[\sP_\mathbf m(W)\simeq\bigoplus_{\mu\in\cP(\mathbf m)}V_\mu.
\]
\end{prp}

\begin{proof}

Recall that $\bfK:=\bfK_n$ is the complexification of $K:=K_n$, and that $\bfM$ is the complexification of $M$. The map $\sP_{\mathbf m}(W)\to C^\infty(X)^{G_r}\simeq C^\infty(Y)$, $f\mapsto \Psi^{-m}f$, is a $K$-equivariant embedding. It follows that $\sP_\mathbf m(W)$ is a direct sum of irreducible $\bfM$-spherical $\bfK$-modules.
Next we determine the multiplicity of every $\bfK$-module $V_\lambda$ in $\sP_\mathbf m(W)$. 
As in Section
\ref{Sec-PaamPart},  $(\mathrm{GL}_n,\mathrm{GL}_r)$-duality entails that $\sP_\mathbf m(W)\simeq E_{\mathbf m}$ as $\bfG_n$-modules.
We consider two separate cases.

\textbf{Case I:} $\F=\R,\qH$. 
For every $\eta:=(\eta_1,\eta_2,\ldots,\eta_r)\in\cP_r$, we define
$\eta^\bullet\in\cP$ as in Section \ref{Sec-PaamPart}. 
By Proposition \ref{thmLIL},
\begin{equation}
\label{multipsumI}
[V_\lambda:E_{\mathbf m}]_{\bfK_n}^{}=
\sum_{\xi\in\cP_r}
[E_{\mathbf m}
:E_{\lambda}\otimes E_{\xi}]_{\bfG_n}^{}.
\end{equation}
If   
$[E_{\mathbf m}
:E_{\lambda}\otimes E_{\xi}]_{\bfG_n}^{}\neq 0$ for some $\xi\in\cP_r$, then
$[E^*_{\mathbf m}
:E^*_{\lambda}\otimes E^*_{\xi}]_{\bfG_n}^{}\neq 0$ and hence
 $\lambda,\xi\subseteq \mathbf m$
(see Remark 
\ref{RmkLiLRChd}). Therefore
\begin{align*}
[E_{\mathbf m}
:E_{\lambda}\otimes E_{\xi}]_{\bfG_n}^{}&=
\lag s_{\mathbf m^\bullet},s_{\lambda^\bullet}s_{\xi^\bullet}
\rag_\Lambda\\
&=\lag s_{\mathbf m^\bullet\bls \xi^\bullet},s_{\lambda^\bullet}\rag_\Lambda=
\lag s_{{(\mathbf m^\bullet\bls \xi^\bullet)}^\circ},s_{\lambda^\bullet}\rag_\Lambda
=\delta_{(\mathbf m\bls\xi)^\circ,\lambda}.
\end{align*}
It follows that  the value of  \eqref{multipsumI} is 0 or 1, with the latter occurring exactly when $\lambda\sseq\mathbf m$.

\textbf{Case II:}  
$\F=\C$. 
Let $\alpha,\beta\in\cP_n$ be defined by
\[
\alpha:=(\underbrace{m,\ldots,m}_{r\text{ times}},0,\ldots,0)\ \text{ and }\ 
\beta:=
(\underbrace{m,\ldots,m}_{n-r\text{ times}},0,\ldots,0).
\]
The statement of the proposition  
is equivalent
(after twisting  $E_\mathbf m$ by $\det(\cdot)^m\otimes  1$) to showing that 
the restriction of the $\mathrm{GL}_n(\C)\times\mathrm{GL}_n(\C)$-module $M_\alpha\otimes M_\beta$ 
 to the diagonal $\mathrm{GL}_n(\C)$ is a multiplicity-free 
direct sum of $\mathrm{GL}_n(\C)$-modules $M_\eta$ for $\eta$ of the form
\begin{equation}
\label{etaformM}
(m+\xi_1,\ldots,m+\xi_r,
\underbrace{m,
\ldots,m}_{n-2r\text{ times}},m-\xi_r,\ldots,m-\xi_1),
\end{equation}
where $\xi:=(\xi_1,\ldots,\xi_r)$ varies through partitions satisfying $\xi_1\leq m$.
Assume that $M_\eta\subseteq M_\alpha\otimes M_\beta$. 
From Remark \ref{RmkLiLRChd}
it follows that 
$\beta\subseteq\eta$, and in addition, every column of the skew Young diagram corresponding to $\eta\bls \beta$ has height at most $r$. Since $r\leq n-r$, it follows that $\eta_i\leq m$ for every $i>n-r$. Consequently, the skew Young diagram corresponding to $\eta\bls \beta$ is a disjoint union of 
two (non-skew) diagrams corresponding to the partitions
\[
\eta^{+}:=
(\eta_1-m,\ldots,\eta_r-m)\text{ and }
\eta^{-}:=(\eta_{n-r+1},\ldots,\eta_n).
\]
 From \cite[Sec. I.5.7]{Macdonald} it follows that $s_\eta=s_{\eta^+}s_{\eta-}$, hence
\begin{align*}
\label{sEsAsB}
[M_\eta:M_\alpha\otimes M_\beta]_{\mathrm{GL}_n(\C)}^{}&=\lag s_\eta,s_{\alpha}s_\beta\rag_\Lambda
=
\lag s_{\eta\bls\beta},s_\alpha\rag_\Lambda
\\
&
=\lag s_{\eta^+}s_{\eta-},s_\alpha\rag_\Lambda =
\lag s_{\eta^+},s_{\alpha\bls\eta^-}\rag_\Lambda
=
\lag s_{\eta^+},
s_{({\alpha\bls\eta^-})^\circ}
\rag_\Lambda
=
\delta_{\eta^+,({\alpha\bls\eta^-})^\circ}
.
\end{align*}
It is now straightforward to verify that 
$[M_\eta:M_\alpha\otimes M_\beta]_{\mathrm{GL}_n(\C)}^{}\leq 1$, with equality occuring if and only if $\eta$ is of the form given in \eqref{etaformM}. 
\end{proof}

\begin{rmk}
\label{RHORHO}
Recall that $\varrho:=(\varrho_1,\ldots,\varrho_r)$ is the vector defined in \eqref{formulRho}.
Let
$\varrho'\in\C^r$ denote the vector $\varrho_{\tau,\alpha}$ defined in \eqref{rhotaualph}, for
$\tau:=\frac{d}{2}$ and $\alpha:=s-\varrho_1$.
Then
the vectors $\varrho$ and $\varrho'$
are related by the relation
\begin{equation*}
\varrho_{r-i+1}+\varrho_i'=s
\text{ for }1\leq i\leq r.
\end{equation*}
\end{rmk}

\begin{prp}
\label{PrpVanCondd}
Let $\lambda\in\cP_r$, and  let 
$d_\lambda(x,s)$ be as in Lemma \ref{clmusdd+r}.  
Let 
$\varrho_{\tau,\alpha}$ be 
as in
\eqref{rhotaualph}
for $\tau:=\frac{d}{2}$ and 
$\alpha:=s-\varrho_1$. For every 
$\nu\in\cP_r$, if $\lambda\nsubseteq\nu$ then
\begin{equation}
\label{dla2v+ddelt+2s-}
d_\lambda
\big(\nu+\textstyle
\varrho_{\tau,\alpha}
,s\big)=0
\ 
\text{for every }s\in\C.
\end{equation}
\end{prp}
\begin{proof}
Fix $\nu\in\cP_r$. Since $d_\lambda(x,s)$ is a polynomial in $x=(x_1,\ldots,x_r)$ and $s$, it is enough to prove 
\eqref{dla2v+ddelt+2s-}
 for infinitely many values of 
$s$. Since $d_\lambda$ is even-symmetric, vanishing of $d_\lambda(x,s)$ at a point $(x_1,\ldots,x_r)$ is equivalent to its vanishing at the point $(-x_r,\ldots,-x_1)$. Therefore it suffices to show that
$
d_\lambda(y,s)=0
$
for $y:=(y_1,\ldots,y_r)$ given by
\[
y_i:=-\nu_{r+1-i}-\frac{d}{2}(i-1)-s+\varrho_1\text{ for every }
1\leq i\leq r.
\]
Next set $s:=-m$ for some integer $m\geq \lambda_1$.  By
Remark \ref{RHORHO} and
 Lemma \ref{clmusdd+r}, 
it suffices to show that
\begin{equation}
\label{EQvCla-m}
c_{\lambda,-m}(\eta)=0\,
\text{ for }\,\eta:=(m-\nu_r,\ldots,m-\nu_1).
\end{equation}
Recall that  $\mathbf m:=(m,\ldots,m)\in\cP_r$ is the partition corresponding to an $r\times m$ Young diagram. 
Set \begin{equation}
\label{dfnofRM}
\sR_\mathbf m(W):=\left\{\Psi^{-m}f\ :\ f\in\sP_\mathbf m(W)\right\},
\end{equation}
where $\Psi$ is defined in Section
\ref{Sec-theOprDlam}.
The map \[
\boldsymbol{j}_m:\sR_\mathbf m(W)\to\sP_\mathbf m(W)\ ,\  f\mapsto \Psi^{m}f
\] is an isomorphism of $\bfK_n$-modules. 
Therefore Proposition \ref{propR^mKdec} implies that $V_\eta$ occurs as a 
$\bfK_n$-submodule of $\sR_\mathbf m(W)$. 
By restriction to $X$, and then factoring $G_r$-invariant functions on $X$ through $Y\simeq X/G_r$, we obtain  an embedding $\sR_\mathbf m(W)\subset C^\infty(Y)$. From
$K_n$-invariance of
 $D_{\lambda,-m}$  it follows that 
$D_{\lambda,-m} \sR_\mathbf m\subseteq \sR_\mathbf m$. Furthermore,  the diagram
\[
\xymatrix{
\sR_\mathbf m(W)\ar[r]^{D_{\lambda,-m}}\ar[d]_{\boldsymbol{j}_m} & \sR_\mathbf m(W)\ar[d]^{\boldsymbol{j}_m}\\
\sP_\mathbf m(W)\ar[r]_{D_\lambda}& \sP_\mathbf m(W)
}
\] 
commutes.
From the latter commutative diagram it follows that, in order to prove \eqref{EQvCla-m}, it suffices  to show that $(D_\lambda\circ\boldsymbol j_m)V_\eta=\{0\}$.
Note that $D_\lambda\in\iota(F_\lambda\otimes F_\lambda^*)$, where $\iota$
is the map 
defined in \eqref{iotasPDAW}. The elements of  $\iota(F_\lambda^*)$ act on $\sP(W)$ as 
$K_n$-invariant constant coefficient differential operators, and by considering the $\bfG_r$-action it follows that they map $\sP_\mathbf m(W)$ into $\sP_{\chi}(W)$,
where 
\[
\chi=(m-\lambda_r,\ldots,m-\lambda_1).
\] Therefore the map 
$D_\lambda:\sP_\mathbf m(W)\to \sP_\mathbf m(W)$ factors through a $K_n$-equivariant map
\begin{equation}
\label{map-PmPchhh}
\sP_\mathbf m(W)\to \sP_\chi(W).
\end{equation}
Suppose that $(D_\lambda\circ\boldsymbol j_m)V_\eta\neq\{0\}$.
Since the map \eqref{map-PmPchhh} is $K_n$-equivariant,  $[V_\eta:\sP_{\chi}(W)]_{\bfK_n}>0$. From
 Proposition \ref{VetaPm+l*} it follows that 
$\lambda\subseteq \nu$, which is a contradiction. 
\end{proof}
\section{Proof of Theorem 
\ref{maintheorem}
}
\label{Sec-pf-Main-Thm}
In this section we complete the proof of  Theorem \ref{maintheorem}. We start by proving the following lemma.
\begin{lem}
\label{lem66.1}
Let $\lambda\in\cP_r$. Then there exists a constant $\gamma_\lambda'\in\C$ such that for every $s\in\C$, and every 
$\mu\in\cP_r$,
\begin{equation}
\label{clas=alaPipla}
c_{\lambda,s}(\mu)=
\gamma_\lambda'
P_\lambda\left(\mu+\varrho;\frac{d}{2},s-\varrho_1\right)
.
\end{equation}
\end{lem}
\begin{proof}
First 
fix $s>\frac{1}{2}\varrho_1$. 
From 
Proposition \ref{PrpVanCondd}
and Lemma \ref{clmusdd+r} we obtain
\begin{equation*}
\textstyle
c_{\lambda,s}\left(\nu+\left(\frac{d}{2}(r-1)+s-2\varrho_1\right)
\mathbbm 1\right)=0\,\text{ if }|\nu|\leq |\lambda|\text{ and }\nu\neq \lambda.
\end{equation*}
From the vanishing part of the statement of 
Proposition \ref{prpuniqPip}
it follows that 
the polynomial 
$
q_\lambda(x):=
P_\lambda\left(
x+
\varrho;\frac{d}{2},s-\varrho_1\right)
$
also vanishes for
all $x:=\nu+\left(\frac{d}{2}(r-1)+s-2\varrho_1\right)\mathbbm 1$, where $\nu$ 
satisfies $|\nu|\leq |\lambda|$ and $\nu\neq \lambda$. The uniqueness part of the statement of 
Proposition \ref{prpuniqPip} now implies 
that there exists a scalar $\gamma_\lambda'\in \C$, possibly depending on $s$,  
such that \eqref{clas=alaPipla} holds for every $\mu$. Since both sides of \eqref{clas=alaPipla} are polynomials in $s$ and $\mu$
(see Lemma \ref{clmusdd+r}), and
\eqref{clas=alaPipla} holds for a
Zariski dense subset of values $(\mu,s)\in \C^r\times \C$, it follows that  
\eqref{clas=alaPipla} indeed holds for every $s\in\C$, and $\gamma_\lambda'$ is a rational function of $s\in\C$.

Next we show that  $\gamma_\lambda'$ does not depend on $s$.
Since the value of $c_{\lambda,s}(\mu)$ is a polynomial 
in $s$ and $\mu$, and  
$P_\lambda(\mu+\varrho;\frac{d}{2},s-\varrho_1)$ is $2\lambda$-monic in $\mu$, it follows that $\gamma_\lambda'$ is a polynomial in $s$. 
However, from the combinatorial formula for $P_\lambda$ that is given in \eqref{formcmbPip}, it follows that the degree of $s$ in
$P_\lambda\left(\mu+\varrho;\frac{d}{2},s-\varrho_1\right)$
 is exactly $2|\lambda|$, whereas the degree of $s$ in $c_{\lambda,s}(\mu)$ is at most $2|\lambda|$.
By comparing the degrees 
of $s$ on both sides of \eqref{clas=alaPipla}, it follows that $\gamma_\lambda'$ is a constant independent of $s$.
\end{proof}


To complete the proof of
Theorem \ref{maintheorem}, we need to 
prove that $\gamma_\lambda'=\gamma_\lambda$, where
$\gamma_\lambda$ is defined in \eqref{defConstgamla}.
The rest of this section is devoted to the proof of the latter claim.

Set $t:=\dim_\R(A)$ and let $v_1,\ldots,v_t$ be an orthonormal basis for $A$ with respect to the pairing $
(\cdot,\cdot)_A^{}$. Set $\varphi_i:=(\cdot,v_i)_A^{}\in A^*$ for every $1\leq i\leq  t$, so that $\varphi_i(v_j)=\delta_{i,j}$.
For every $1\leq i\leq t$,
let $\partial_{v_i}\in\sD(A)$ denote the directional derivative corresponding to $v_i$, so that $\partial_{v_i}(\varphi_j)=\varphi_j(v_i)=\delta_{i,j}$.  
For every $1\leq i\leq t$, set $q_i:=\mapvWA^*(\varphi_i)$
where $\mapvWA^*$ is as in \eqref{q*PAPW}, and let $\partial_{q_i}\in \sD(W)$ denote the second-order differential operator  corresponding to $q_i$ under the isomorphism $\sP(W)\simeq \sD(W)$  defined in
\eqref{PDWPDA}.

Fix an integer $m\geq 1$ and set
\begin{equation}
\label{D(m)}
D^{(m)}:=\sum_{|\lambda|=m}\widetilde D_\lambda,
\end{equation}
where $\widetilde D_\lambda$ is defined in \eqref{DtildeLam}.
Then $D^{(m)}$ acts on every $K$-module $V_\mu\subseteq C^\infty(Y)\simeq C^\infty(X)^{G_r}$ by the scalar \[
c^{(m)}(\mu):=\sum_{|\lambda|=m}c_{\lambda,0}(\mu).
\] 
From the definition of the operators $\widetilde{D}_\lambda$ it follows that if  
$\{\widetilde{v}_j\}$ is a basis for $\sP^m(A)$ and
$\{\widetilde{\partial}_j\}$ is the corresponding dual basis 
for $\sD^m(A)\simeq \sP^m(A)^*$, then 
$D^{(m)}=\sum_j \iota(\widetilde{v}_j)
\iota(\widetilde{\partial}_j)$
where $\iota$ is as in \eqref{iotasPDAW}. In particular,
if we choose the basis  
\[
\big\{(\partial_{v_1})^{m_1}\cdots (\partial_{v_t})^{m_t}\, :\, m_1+\cdots+m_t=m\big\}
\]
for $\sD^m(A)$, then 
the corresponding dual basis for $\sP^m(A)$ will be
\[
\left\{
\frac{1}{m_1!\cdots m_t!}\varphi_1^{m_1}\cdots \varphi_t^{m_t}
\, :\, m_1+\cdots+m_t=m
\right\},
\]
and thus
\begin{align*}
D^{(m)}&=\frac{1}{m!}\sum_{m_1+\cdots+m_t=m}{m\choose m_1,\ldots,m_t}q_1^{m_1}\cdots q_t^{m_t}(\partial_{q_1})^{m_1}\cdots
(\partial_{q_t})^{m_t}\\
&=
\frac{1}{m!}(q_1\partial_{q_1}+\cdots+q_t\partial_{q_t})^m+D'
=
\frac{1}{m!}\big(D^{(1)}\big)^m+D',
\end{align*}
where $D'$ is a differential operator with order strictly less than the order of $D$.


Let $\g g_\C$ be the Lie algebra of $\bfG_n$, and let
$\Omega_\g k$ and $\Omega_\g g$ denote the Casimir operators of 
$\g k_\C$ and $\g g_\C$ (see the Appendix  for more information).
Also, Let $\mathbf E$ be the degree (or Euler) operator on $W$. The operator $D^{(1)}$ is a 
polynomial-coefficient differential operator on $W_\C$.
The next proposition  expresses $D^{(1)}$  in terms of $\Omega_{\g k}$, $\Omega_{\g g}$, and $\mathbf E$.
\begin{prp}
\label{prpAppndx}
Let  $D^{(1)}$ be defined as in \eqref{D(m)}. Then
\begin{equation}
\label{CalcOpD(1)}
D^{(1)}=\begin{cases}
-2(n-2)\Omega_{\g k}+\Omega_{\g g}-\mathbf E&\text{ if }\F=\R,\\
-2\Omega_{\g k}+2\Omega_{\g g}+(2n-2r)\mathbf E&\text{ if }\F=\C,\\
-8(2n+1)\Omega_{\g k}+2\Omega_{\g g}+2(2n-2r+1)\mathbf E&\text{ if }\F=\qH.
\end{cases}
\end{equation}
\end{prp}
\begin{proof}
The proof is by a tedious calculation and is deferred to the Appendix.
\end{proof}

\begin{lem}
\label{lem-ev-D(1)} The scalar $c^{(1)}(\mu)$ is a quadratic polynomial in $\mu_1,\ldots,\mu_r$ with top-degree homogeneous part equal to $-4\left(\mu_1^2+\cdots+\mu_r^2\right)$. 
\end{lem}
\begin{proof}
We use the expressions for $D^{(1)}$ given in Proposition \ref{prpAppndx}.
As is well known (e.g., see \cite[Lem. 3.3.8]{GW09}), the action of $\Omega_\g k$ on  $V_\mu$ is by a scalar which is a quadratic polynomial in $\mu_1,\ldots,\mu_n$, whose top homogeneous term is
given by
\begin{equation}
\label{ActionOkk}
\begin{cases}
\frac{2}{n-2}(\mu_1^2+\cdots+\mu_r^2)&\text{ if }\F=\R,\\
2(\mu_1^2+\cdots+\mu_r^2)&\text{ if }\F=\C,\\
\frac{1}{4n+2}(\mu_1^2+\cdots+\mu_r^2)&\text{ if }\F=\qH.\\
\end{cases}
\end{equation}
Fix $m\in\Z$ such that  $m\geq \mu_1$, and let $\sR_\mathbf m(W)$ be as in \eqref{dfnofRM}. 
From Proposition \ref{propR^mKdec} it follows that $[V_\mu:\sR_\mathbf m(W)]_{\bfK_n}^{}>0$. Since elements of $\sR_\mathbf m(W)$ are homogeneous of degree zero, the degree operator $\mathbf E$ vanishes  on $V_\mu$. 
Furthermore, we have $V_\mu\subseteq C^\infty(Y)_{K_n\text{-finite}}^{}$, and $Y\simeq G_n/P_{r,n}$ where $P_{r,n}$ is the $(r,n-r)$ parabolic subgroup of $G_n$. Thus, 
$C^\infty(G_n/P_{r,n})_{K_n\text{-finite}}^{}$ is the space of $K_n$-finite vectors of a degenerate principal series representation of $G_n$ induced from $P_{r,n}$, and the operator $\Omega_\g g$ acts on 
the latter space by a scalar that is independent of $\mu$ (see \cite[Prop. 8.22]{Knapp}). Consequently, the top-degree homogeneous part of $c^{(1)}(\mu)$ is determined by the action of $\Omega_\g k$.
The  lemma now follows from
\eqref{CalcOpD(1)} and 
\eqref{ActionOkk}. 
  \end{proof}
Lemma \ref{lem-ev-D(1)} implies that
the action of $D^{(m)}$ on $V_\mu$ is by a polynomial in $\mu_1,\ldots,\mu_r$ of degree $2m$, whose top-degree homogeneous part is 
$
\frac{(-4)^m}{m!}(\mu_1^2+\cdots+\mu_r^2)^m.
$
On the other hand, from \cite[Eq. (7.3)]{Koornwinder}
it follows that for every $\lambda$ such that $|\lambda|=m$, the top-degree homogeneous part of $c_{\lambda,0}(\mu)$ is up to a scalar equal to $\mathsf  P_\lambda(\mu^2,\frac{d}{2})$,
where $\mathsf P_\lambda(x,\tau)$ is the $\lambda$-monic Jack polynomial and $\mu^2:=(\mu_1^2,\ldots,\mu_r^2)$.
Let $J_\lambda$ denote the normalization of the Jack polynomial introduced in 
\cite[Thm 1.1]{StanleyP}. The scalar relating  $J_\lambda$ and $\mathsf P_\lambda$ is given in \cite[Chap. VI, Eq. (10.22)]{Macdonald} (see also \cite[Thm 5.6]{StanleyP}).
From \cite[Prop. 2.3]{StanleyP}, and the relation between $\mathsf P_\lambda$ and 
$J_\lambda$, it follows that
\begin{equation}
\label{SeqSJJ}
(\mu_1^2+\cdots+\mu_r^2)^m=\sum_{|\lambda|=m}
\left(\frac{d}{2}\right)^m
\frac{m!}{\prod_{\flat\in\lambda}
\left(
\frac{d}{2}(a_\lambda(\flat)+1)
+l_\lambda(\flat)
\right)}
\mathsf  P_\lambda\left(\mu^2,\frac{d}{2}\right).
\end{equation}
Since the polynomials $\mathsf{P}_\lambda(\mu^2,\frac{d}{2})$ are linearly independent, by considering the top-degree homogeneous parts of both sides of \eqref{SeqSJJ} it follows that for every $\lambda\in\cP_r$ such that $|\lambda|=m$,
\[
c_{\lambda,0}(\mu)=
\frac{(-2d)^m}{\prod_{\flat\in\lambda}
\left(
\frac{d}{2}(a_\lambda(\flat)+1)
+l_\lambda(\flat)
\right)}
P_\lambda\left(\mu+\varrho;\frac{d}{2},-\varrho_1\right).
\]
Lemma \ref{lem66.1}  completes the proof of  Theorem \ref{maintheorem}.


\section*{Appendix: Proof of Proposition \ref{prpAppndx}}
In this Appendix, we exhibit the details of the calculations that yield the formulas \eqref{CalcOpD(1)}
for the operator $D^{(1)}$. 

Let $\kappa:\g k_\C\times\g k_\C\to\C$ denote the invariant bilinear form which is equal to the Killing form of $\g k_\C$ when $\F=\R$ or $\qH$, and is given by $\kappa(x,y):=\tr(xy)$ when $\F=\C$.
The Casimir operator of $\g k_\C$ is  $\Omega_\g k:=\sum_{i=1}^{\dim\g k_\C}x_ix^i$, where $\{x_i\,:\,1\leq i\leq \dim\g k_\C\}$
is a basis for $\g k_\C$, and $\{x^i\,:\,1\leq i\leq \dim\g k_\C\}$ is the corresponding dual basis with respect to $\kappa(\cdot,\cdot)$.
We define the Casimir operator $\Omega_\g g$
of $\g g_\C$ similarly. Explicit formulas for $\Omega_\g g$ are given in 
\eqref{CASgl(n))}, \eqref{CASgl(nC))}, and \eqref{CASgl(nH))}.
In the following, $\sfE_{i,j}$ will always denote a matrix with a 1 in the $(i,j)$ position and 0's elsewhere (the number of rows and columns of 
$\sfE_{i,j}$ will be clear from the context). 

\textbf{Case I:} $\F=\R$. Recall that in this case $A=\mathrm{Sym}_{r\times r}(\R)$. The  orthonormal basis of 
$A$ with respect to $(\cdot,\cdot)_A^{}$ is
\[
\sfE_{i,i}\text{ for }1\leq i\leq r\ \text{ and }\ 
\frac{1}{\sqrt{2}}(\sfE_{i,j}+\sfE_{j,i})\text{ for }
1\leq i< j\leq r.
\]
We fix generators 
$
x_{i,j}\in\sP(A)
$, where $1\leq i\leq j\leq r$,
and 
$y_{i,j}\in\sP(W)$, where $1\leq i\leq n$ and $1\leq j\leq r$,
such that
$x_{i,j}([t_{a,b}]):=t_{i,j}$
for every matrix
$[t_{a,b}]\in A$,  and 
$y_{i,j}([t_{a,b}]):=t_{i,j}$ for  every matrix
$[t_{a,b}]\in W$.
The isomorphism $\iota_A^{}:A\to A^*\simeq\sP^1(A)$ that is defined in
\eqref{iotaAWs} is given by
\[
\sfE_{i,i}\mapsto x_{i,i}\text{ for }
1\leq i\leq r
\ \text{ and }\ 
\frac{1}{\sqrt{2}}(\sfE_{i,j}+\sfE_{j,i})\mapsto \sqrt{2}x_{i,j}\text{ for }
1\leq i<j\leq r.
\]
The map 
$\mapvWA^*
:\sP(A)\to \sP(W)
$ of \eqref{q*PAPW} is given by
\[
x_{a,a}\mapsto \sum_{i=1}^ny_{i,a}^2
\text{ for }1\leq a\leq r
\ \text{ and }\ 
\sqrt{2}x_{a,b}\mapsto \sqrt{2}\sum_{i=1}^r y_{i,a}y_{i,b}\text{ for }1\leq a<b\leq r.
\]
Finally, the isomorphism $\iota_W^{}:W\simeq W^*\simeq \sD^1(W)$ of
\eqref{iotaAWs} is given by $
y_{i,a}\mapsto \partial_{i,a}:=\frac{\partial}{\partial{y_{i,a}}}$.
From all of the above, it follows that
\begin{equation}
\label{D(1)R}
D^{(1)}
=\sum_{a=1}^r 
\left(\sum_{i=1}^n y_{i,a}^2\right)
\left(\sum_{j=1}^n \partial_{j,a}^2\right)
+\sum_{1\leq a\neq b\leq r}
\left(\sum_{i=1}^n y_{i,a}y_{i,b}\right)
\left(\sum_{j=1}^n \partial_{j,a}\partial_{j,b}\right).
\end{equation}
The embedding 
 $\g k_\C\into\g{gl}_n(\C)$
 gives the realization of $\g k_\C$ as
\begin{equation*}
\g k_\C=\left\{x\in\mathrm{Mat}_{n\times n}(\C)\ :\ x+x^T=0\right\},
\end{equation*}
where $x^T$ is the transpose of $x$. 
Recall that by the definition of $\Omega_\g g$, \begin{equation}
\label{CASgl(n))}
\Omega_\g g=\sum_{1\leq i,j\leq n}
\sfE_{i,j}\sfE_{j,i}.
\end{equation}
The Killing form of $\g k_\C$ is  $\kappa(x,y):=(n-2)\tr(xy)$, and 
therefore 
\begin{equation}
\label{CasimirR}
\Omega_{\g k}=-\frac{1}{2(n-2)}
\sum_{1\leq i<j\leq n}
(\sfE_{i,j}-\sfE_{j,i})^2.
\end{equation}
The action of every $\sfE_{i,j}\in \g{gl}_n(\C)$ on $\sP(W)$, corresponding to the derived action of $\bfG_n$, is given by polarization operators, that is 
\begin{equation}
\label{POLIZR}
\sfE_{i,j}=-
\sum_{a=1}^ry_{j,a}\partial_{i,a}.
\end{equation}
Finally, the action of the degree operator is given by $\mathbf E=\sum_{i=1}^n\sum_{j=1}^ry_{i,j}\partial_{i,j}
$. Thus, expanding the right hand side of 
\eqref{D(1)R}, and using \eqref{CASgl(n))}, \eqref{CasimirR},  and \eqref{POLIZR}, we obtain 
$D^{(1)}=-2(n-2)\Omega_{\g k}+\Omega_{\g g}-\mathbf E$.

\textbf{Case II:} $\F=\C$.
In this case $A$ is the space of $r\times r$ complex Hermitian matrices. The orthonormal basis of 
$A$ with respect to $(\cdot,\cdot)_A^{}$ is
\[
\sfE_{i,i}\text{ for }1\leq i\leq r,\ \,
\frac{1}{\sqrt{2}}(\sfE_{i,j}+\sfE_{j,i})\text{ for }
1\leq i<j\leq r,\,\text{ and }
\frac{\sqrt{-1}}{\sqrt{2}}(\sfE_{i,j}-\sfE_{j,i})
\text{ for }1\leq i<j\leq r.
\]
Moreover, $W=\mathrm{Mat}_{n\times r}(\C)$, considered as a real vector space. Fix generators $x_{i,j},y_{i,j}\in\sP(W)$, $1\leq i\leq n$, $1\leq j\leq r$ for the algebra  $\sP(W)$, such that 
\[
x_{i,j}([t_{a,b}]):=\Re(t_{i,j})\text{ and }
y_{i,j}([t_{a,b}]):=\Im(t_{i,j}).
\]
The map 
$\mapvWA^*
:\sP(A)\to \sP(W)
$ of \eqref{q*PAPW} is given by
\[
\mapvWA^*(\iota_A^{}(S)):=
\begin{cases}
\sum_{a=1}^nx_{a,i}^2+y_{a,i}^2 &
\text{ if }S=\sfE_{i,i}\text{ where }1\leq i\leq r,\\[1mm] \sqrt{2}\sum_{a=1}^n
(x_{a,i}x_{a,j}+y_{a,i}y_{a,j})
&\text{ if }
S=\frac{1}{\sqrt{2}}(\sfE_{i,j}+\sfE_{j,i})
\text{ where }1\leq i<j\leq r,\\[1mm]
\sqrt{2}\sum_{a=1}^n
(-x_{a,j}y_{a,i}+x_{a,i}y_{a,j})
&\text{ if }
S=\frac{\sqrt{-1}}{\sqrt{2}}(\sfE_{i,j}-\sfE_{j,i})
\text{ where }1\leq i<j\leq r
.
\end{cases}
\]
For the realization of the derived action of 
$\bfG_n\simeq\mathrm{GL}_n(\C)\times \mathrm{GL}_n(\C)$
on $\sP(W)$  it will be more convenient to work with the coordinates $z_{i,j}$ and $\xi_{i,j}$ on $W_\C\simeq
\mathrm{Mat}_{n\times r}(\C)\oplus \mathrm{Mat}_{n\times r}(\C)$,  
where  $1\leq i\leq n$ and 
$1\leq j\leq r$,
given by
\[
z_{i,j}:=x_{i,j}-\sqrt{-1}y_{i,j}\,\text{ and }\,
\xi_{i,j}:=x_{i,j}+
\sqrt{-1}y_{i,j}.
\]
From these formulas it follows that
$
\frac{\partial}{\partial x_{i,j}}=
\frac{\partial}{\partial z_{i,j}}+
\frac{\partial}{\partial \xi_{i,j}}
$ and 
$
\frac{\partial}{\partial y_{i,j}}=
\sqrt{-1}
\left(\frac{\partial}{\partial \xi_{i,j}}-
\frac{\partial}{\partial z_{i,j}}\right)
$.

Set 
$\partial_{\xi_{i,j}}:=\frac{\partial}{\partial \xi_{i,j}}$
and 
$\partial_{z_{i,j}}:=\frac{\partial}{\partial z_{i,j}}$. By a direct calculation, we obtain 
\begin{align*}
D^{(1)}&=4\sum_{i=1}^r
\left(\sum_{a=1}^n
z_{a,i}\xi_{a,i}\right)
\left(\sum_{a=1}^n
\partial_{z_{a,i}}
\partial_{\xi_{a,i}}\right)\\
&+
2\sum_{1\leq i<j\leq r}
\left(
\sum_{a=1}^n
z_{a,i}\xi_{a,j}+z_{a,j}\xi_{a,i}
\right)
\left(
\sum_{a=1}^n
\partial_{z_{a,i}}\partial_{\xi_{a,j}}+
\partial_{z_{a,j}}\partial_{\xi_{a,i}}
\right)\\
&+2
\sum_{1\leq i<j\leq r}\left(
\sum_{a=1}^n
z_{a,i}\xi_{a,j}-z_{a,j}\xi_{a,i}
\right)
\left(
\sum_{a=1}^n
\partial_{z_{a,i}}\partial_{\xi_{a,j}}-
\partial_{z_{a,j}}\partial_{\xi_{a,i}}
\right).
\end{align*}
Next we describe the derived action of $\bfG_n$ on $\sP(W)$. 
The Lie algebra of $\bfG_n$ is isomorphic to the direct sum $\g{gl}_n(\C)\oplus\gl_n(\C)$.
We will denote the matrices in the standard bases of the left and right summands of this direct sum by  $\sfE_{i,j}^{(l)}$ and  
$\sfE_{i,j}^{(r)}$.
Recall that $\sP(W)$ can be 
identified with polynomials on the complex vector space $W_\C$. 
The actions of $\sfE_{i,j}^{(l)}$ and $\sfE_{i,j}^{(r)}$ on 
$\sP(W)$ are  by 
polarization operators
$
\sum_{a=1}^r
z_{i,a}
\partial_{z_{j,a}}
$ and $
-\sum_{a=1}^r\xi_{j,a}
\partial_{\xi_{i,a}}
$.
The embedding of $\g k_\C$ into the Lie algebra of $\bfG_n$ is the diagonal map
$
\g{gl}_n(\C)\into
\g{gl}_n(\C)\oplus\g{gl}_n(\C)
$. 
The  Casimir 
operators 
$\Omega_\g g$ and $\Omega_\g k$
are
 \begin{equation}
 \label{CASgl(nC))}
\Omega_\g g=\sum_{1\leq i,j\leq n}
\sfE_{i,j}^{(l)}\sfE_{j,i}^{(l)}
+
\sfE_{i,j}^{(r)}\sfE_{j,i}^{(r)}
\,\text{ and }\,
\Omega_{\g k}=
\sum_{1\leq i,j\leq n}
\left(
\sfE_{i,j}^{(l)}+\sfE_{i,j}^{(r)}
\right)
\left(
\sfE_{j,i}^{(l)}+\sfE_{j,i}^{(r)}
\right).
\end{equation}
Finally, the degree operator is
$
\mathbf E=\sum_{i=1}^n\sum_{j=1}^r
\left(z_{i,j}\partial_{z_{i,j}}+
\xi_{i,j}\partial_{\xi_{i,j}}
\right)
$.
With a calculation similar to Case I, we obtain
$
D^{(1)}=-2\Omega_{\g k}+2\Omega_{\g g}+(2n-2r)\mathbf E
$.

\textbf{Case III:} $\F=\qH$. 
The calculations are similar to the previous cases, only more elaborate. 
In this case, $A$ is the space of 
hermitian quaternionic matrices. Let $\{1,\mathbf i,\mathbf j,\mathbf k\}$ denote the standard $\R$-basis of $\qH$.
Then the  orthonormal basis of $A$ consists of 
$\sfE_{a,a}$, for $1\leq a\leq r$, and 
\[
\frac{1}{\sqrt{2}}(\sfE_{a,b}+\sfE_{b,a}),\,
\frac{\mathbf i}{\sqrt{2}}(\sfE_{a,b}-\sfE_{b,a}),\,
\frac{\mathbf j}{\sqrt{2}}(\sfE_{a,b}-\sfE_{b,a}),
\text{ and }
\frac{\mathbf k}{\sqrt{2}}(\sfE_{a,b}-\sfE_{b,a}),
\text{ for }1\leq a<b\leq r.
\]
Similar to Case II, it will be easier to use coordinates on $W_\C$. The embedding 
$W\into W_\C$ 
is given in matrix form by
\[
\mathrm{Mat}_{n\times r}(\qH)\into \simeq\mathrm{Mat}_{2n\times 2r}\ ,\
\Aa+\mathbf i\Bb+\mathbf j\Cc+\mathbf k\Dd
\mapsto 
\begin{bmatrix}
\Aa+\sqrt{-1}\Bb & -\Cc-\sqrt{-1}\Dd\\
\Cc-\sqrt{-1} \Dd& \Aa-\sqrt{-1}\Bb
\end{bmatrix},
\]
where $\Aa,\Bb,\Cc,\Dd\in\mathrm{Mat}_{n\times r}(\R)$.
The new coordinates on $W_\C$, and the relations between the corresponding  directional derivatives are as follows.
\begin{equation}
\label{newcoorqH}
\begin{cases}
\xi_{i,j}:=\Aa_{i,j}+\sqrt{-1}\Bb_{i,j},\\
\xi_{i+n,j+r}:
=\Aa_{i,j}-\sqrt{-1}\Bb_{i,j},\\
\xi_{i+n,j}:=\Cc_{i,j}-\sqrt{-1}\Dd_{i,j},\\
\xi_{i,j+r}:=
-\Cc_{i,j}-\sqrt{-1}\Dd_{i,j}.
\end{cases}
\text{ and }
\begin{cases}
\partial_{\Aa_{i,j}}=\partial_{\xi_{i,j}}+\partial_{\xi_{i+n,j+r}},\\
\partial_{\Bb_{i,j}}=\sqrt{-1}\left(
\partial_{\xi_{i,j}}-\partial_{\xi_{i+n,j+r}}\right),\\
\partial_{\Cc_{i,j}}=\partial_{\xi_{i+n,j}}-\partial_{\xi_{i,j+r}},\\
\partial_{\Dd_{i,j}}=-\sqrt{-1}\left(
\partial_{\xi_{i+n,j}}+
\partial_{\xi_{i,j+r}}\right).
\end{cases}
\end{equation}
Next we define
\[
\Phi_e(a,b):=
\begin{cases}
\sum_{i=1}^n
\xi_{i,a}\xi_{i+n,b+r}&\text{ if }e=1,\\
\sum_{i=1}^n
\xi_{i,a}\xi_{i+n,b}&\text{ if }e=2,\\
\sum_{i=1}^n
\xi_{i+n,a}\xi_{i,b+r}&\text{ if }e=3,\\
\sum_{i=1}^n
\xi_{i,a+r}\xi_{i+n,b+r}&\text{ if }e=4.
\end{cases}
\]
In the coordinates defined in \eqref{newcoorqH}, the map 
$\mapvWA^*
:\sP(A)\to \sP(W)
$ of \eqref{q*PAPW} is given by
\[
\mapvWA^*(\iota_A(S)):=
\begin{cases}
\Phi_1(a,a)-\Phi_3(a,a)
& \text{ if }S=\sfE_{a,a},\\
\frac{1}{\sqrt{2}}\left(
\Phi_1(a,b)+\Phi_1(b,a)-\Phi_3(a,b)-\Phi_3(b,a)\right)
& \text{ if }
S=\frac{1}{\sqrt{2}}(\sfE_{a,b}+\sfE_{b,a}),\\
\frac{\sqrt{-1}}{\sqrt{2}}\left(
\Phi_1(a,b)-\Phi_1(b,a)+\Phi_3(b,a)-\Phi_3(a,b)
\right)
&
\text{ if }S=\frac{\mathbf i}{\sqrt{2}}(\sfE_{a,b}-\sfE_{b,a}),\\
\frac{1}{\sqrt{2}}\left(
\Phi_2(a,b)-\Phi_4(b,a)-\Phi_2(b,a)+\Phi_4(a,b)
\right)
&
\text{ if }S=\frac{\mathbf j}{\sqrt{2}}(\sfE_{a,b}-\sfE_{b,a}),\\
\frac{1}{\sqrt{2}}\left(
\Phi_2(a,b)+\Phi_4(b,a)-\Phi_2(b,a)-\Phi_4(a,b)
\right)
&
\text{ if }S=\frac{\mathbf k}{\sqrt{2}}(\sfE_{a,b}-\sfE_{b,a}).
\end{cases}
\]
For $1\leq e\leq 4$, set 
$\partial\Phi_e(a,b):=
(\iota_W^{})^{-1}\left(\Phi_e(a,b)\right)$, where
$\iota_W:\sD(W)\to \sP(W)$ is the isomorphism given in \eqref{PDWPDA}. In fact $\partial\Phi_e(a,b)$ is the constant coefficient differential operator obtained from $\Phi_e(a,b)$ by substitution of each variable $\xi_{i,j}$ by the corresponding partial derivative 
$\partial_{\xi_{i,j}}$. Then
\begin{align*}
D^{(1)}&=
4\sum_{1\leq a,b\leq r}\left(\Phi_1(a,b)-\Phi_3(a,b)\right)
\left(\partial\Phi_1(a,b)-\partial\Phi_3(a,b)\right)\\
&+
2\sum_{1\leq a,b\leq r}\left(\Phi_2(a,b)-\Phi_2(b,a)\right)
\left(\partial\Phi_2(a,b)-\partial\Phi_2(b,a)\right)\\
&+2\sum_{1\leq a,b\leq r}\left(\Phi_4(a,b)-\Phi_4(b,a)\right)
\left(\partial\Phi_4(a,b)-\partial\Phi_4(b,a)\right).
\end{align*}
The embedding of $\g k_\C$ into the Lie algebra of $\bfG_n$ gives the realization of $\g k_\C$ as 
\[
\g k_\C:=\left\{
x\in\mathrm{Mat}_{2n\times 2n}(\C)\,:\,
x^TJ_n=-J_nx
\right\},
\]
where $J_n$ is as in \eqref{eqMtrxJ}, and $x^T$ denotes the transpose of the matrix $x$. We will denote the matrices in the 
standard basis of $\g{gl}_{2n}(\C)$ by $\sfE_{i,j}$'s.
The Killing form of $\g k_\C$ is $\kappa(x,y):=(4n+2)\tr(xy)$, and consequently, 
\begin{align*}
\Omega_\g k&=
\frac{1}{4(2n+1)}\sum_{1\leq p,q\leq n}\left(\sfE_{p,q}-\sfE_{q+n,p+n}\right)
\left(\sfE_{q,p}-\sfE_{p+n,q+n}\right)\\
&+
\frac{1}{4(2n+1)}
\sum_{1\leq p<q\leq n}\left(\sfE_{p,q+n}+\sfE_{q,p+n}\right)
\left(\sfE_{q+n,p}+\sfE_{p+n,q}\right)
+
\frac{1}{2(2n+1)}
\sum_{1\leq p\leq n}\sfE_{p,p+n}\sfE_{p+n,p}\\
&+
\frac{1}{4(2n+1)}
\sum_{1\leq p<q\leq n}\left(\sfE_{p+n,q}+\sfE_{q+n,p}\right)
\left(\sfE_{q,p+n}+\sfE_{p,q+n}\right)
+\frac{1}{2(2n+1)}\sum_{1\leq p\leq n}\sfE_{p+n,p}\sfE_{p,p+n}.
\end{align*}
As in Cases I and II,  
\begin{equation}
\label{CASgl(nH))}
\Omega_\g g=\sum_{1\leq p,q\leq 2n}\sfE_{p,q}\sfE_{q,p}.
\end{equation} The action of $\sfE_{p,q}$ on $\sP(W)$ is by the polarization operator 
$
-\sum_{i=1}^{2r}\xi_{q,s}
\partial_{\xi_{p,s}}
$.
Finally, the degree operator is 
$
\mathbf E=\sum_{i=1}^{2n}\sum_{j=1}^{2r}\xi_{i,j}
\partial_{\xi_{i,j}}
$.
From all of the above, and by a straightforward calculation, we obtain
$D^{(1)}=-8(2n+1)\Omega_{\g k}+2\Omega_{\g g}+2(2n-2r+1)\mathbf E$.

\end{document}